\documentclass[12pt]{article}
\usepackage{amsmath,amsfonts,amssymb,amsthm,color}

\usepackage{amsmath, amsfonts, amssymb, amsthm, graphicx, color,caption}
\usepackage{epsfig}
\usepackage{fancyhdr}
\usepackage{framed}
\usepackage{floatrow}
\usepackage{mdframed}
\usepackage{url}
\usepackage{cfr-lm}
\usepackage{color}

\usepackage{imakeidx}
\makeindex

\newcounter{abci}\renewcommand{\theabci}{\alph{abci}}

\theoremstyle{plain}
\newtheorem{lem}{Lemma}
\newtheorem*{lem*}{Lemma}
\newtheorem{thm}{Theorem}
\newtheorem{prop}{Proposition}
\newtheorem*{prop*}{Proposition}
\newtheorem{assum}{Assumption}
\newtheorem{Def}{Definition}
\newtheorem{cor}{Corollary}
\newtheorem*{cor*}{Corollary}

\theoremstyle{definition}
\newtheorem{ex}{Example}

\newtheorem{rem}{Remark}

\def \cT {\mathcal{T}}
\def \cR {\mathcal{R}}
\def\cN{{\mathcal{N}}}

\def \cH {\mathcal{H}}

\def\be{\begin{equation}}
\def\ee{\end{equation}}

\title{Invariance and attraction properties \\of Galton-Watson trees}

\author{Yevgeniy Kovchegov\footnote{Department of Mathematics, Oregon State University, Corvallis, OR 97331-4605} 
\quad and \quad 
Ilya Zaliapin\footnote{Department of Mathematics and Statistics, University of Nevada Reno, NV 89557-0084}}

\date{}

\begin{document}

\maketitle

\begin{abstract}
We give a description of invariants and 
attractors of the critical and subcritical Galton-Watson tree measures under the 
operation of Horton pruning (cutting tree leaves with subsequent series reduction). 
Under a regularity condition, the class of invariant measures consists of
the critical binary Galton-Watson tree and a one-parameter family of critical Galton-Watson trees with 
offspring distribution $\{q_k\}$ that has a power tail 
$q_k\sim Ck^{-(1+1/q_0)}$, where $q_0\in(1/2,1)$.  
Each invariant measure has a non-empty domain of attraction under consecutive 
Horton pruning, specified by the tail behavior of the initial 
Galton-Watson offspring distribution.
The invariant measures satisfy the Toeplitz property for the
Tokunaga coefficients and obey the Horton law with exponent $R = (1-q_0)^{-1/q_0}$.
\end{abstract}


\section{Introduction and motivation}
The study of random trees invariant with respect to combinatorial pruning 
(erasure) from leaves down to the root emerges in attempts to understand symmetries
of natural trees observed in fields as diverse as hydrology, phylogenetics, or computer science.  
In addition, it provides a unifying framework for analysis of
coalescence and annihilation dynamical models, including the celebrated Kingman's 
coalescent, and self-similar stochastic processes on the real line; 
see a recent survey \cite{KZsurvey} for details.
A special place in the invariance studies is occupied by the family of Galton-Watson trees,
whose transparent generation mechanism makes it a convenient testbed for
general theories and approaches.  
A Galton-Watson tree describes the trajectory of the Galton-Watson
branching process \cite{AN_book} with a single progenitor and offspring distribution
$\{q_k\}$, $k=0,1,\dots$.
We write $\mathcal{GW}(q_k)$ for the probability measure that corresponds to this random tree.
A tree is called {\it critical} if the expected progeny of a single member equals unity: 
$\sum_{k=1}^\infty kq_k=1$. 
Similarly, a tree is {\it subcritical} if $\sum_{k=1}^\infty kq_k<1$. 
In this paper we analyze the invariance and attraction properties 
of critical and subcritical Galton-Watson trees under the operation of {\it combinatorial 
Horton pruning} -- cutting tree leaves and their parental edges followed by series reduction 
(removing vertices of degree $2$).
The Horton pruning (formally introduced in Sect.~\ref{sec:HortonPruning} and illustrated 
in Fig.~\ref{fig:HS}) is a discrete, combinatorial analog of the continuous 
{\it erasure} or {\it trimming} studied by Neveu \cite{Neveu86}, Neveu and Pitman \cite{NP,NP2},
Le Jan \cite{LeJan91}, Evans \cite{Evans2005}, 
and Evans, Pitman and Winter \cite{EPW06}.

\subsection{Invariance}
Combinatorial prune invariance of critical and subcritical Galton-Watson trees  was first examined by Burd~et~al.~\cite{BWW00}, under 
the assumption of a finite second moment for the offspring distribution, $\sum_{k=1}^\infty k^2q_k<\infty$.
These authors have shown that the only invariant measure in this class 
corresponds to the critical binary Galton-Watson tree, $q_0=q_2=\frac{1}{2}$ 
\cite[Thm.~ 3.9]{BWW00}. 
Here we substantially relax the regularity constraint on the offspring distribution;
see Assumption \ref{asm:reg} and Lem.~\ref{lem:RegCond} in Sect.~\ref{sec:reg}.
This reveals an abundance of prune-invariant measures with infinite second moment. 
Theorem \ref{thm:completeGW} describes all such measures among the critical 
and subcritical Galton-Watson trees that satisfy Assumption \ref{asm:reg}.
This infinite family of {\it Invariant Galton-Watson (IGW)} measures can be 
characterized by a single parameter -- the probability $q_0 \in [1/2,1)$ of having no offsprings. 
An individual distribution from this family is denoted by $\mathcal{IGW}(q_0)$; it
is a critical distribution with the offspring generating function 
$$Q(z)=\sum\limits_{k=0}^\infty q_kz^k = z+q_0(1-z)^{1/q_0}.$$
The case $q_0=1/2$ with $Q(z)=(1+z^2)/2$ corresponds to the critical binary Galton-Watson tree $\mathcal{IGW}(1/2)=\mathcal{GW}(q_0\!=q_2\!=\!1/2)$.
Every invariant Galton-Watson measure  $\mathcal{IGW}(q_0)$ with $q_0\in(\frac{1}{2},1)$ corresponds to an unbounded offspring distribution 
of Zipf type with infinite second moment:
\[q_k \sim C k^{-(1+1/q_0)}\quad\text{as}\quad k\to\infty.\]

\subsection{Attraction}
Burd et al.~\cite[Thm.~ 3.11]{BWW00} have shown that any 
critical Galton-Watson tree with a bounded offspring number 
(there exists $b$ such that $q_k=0$ for all $k\ge b$) converges 
to the critical binary Galton-Watson tree under iterative Horton pruning, conditioned on surviving under the pruning.
Our Thm.~\ref{thm:IGWattractor} shows that the collection of $\mathcal{IGW}(q_0)$ 
measures for $q_0\in[1/2,1)$ and a point measure $\mathcal{GW}(q_0\!=\!1)$ are 
the only possible attractors of critical and subcritical Galton-Watson measures that
satisfy Assumption \ref{asm:reg}, with
respect to the iterative Horton pruning.
Specifically, all subcritical measures converge to $\mathcal{GW}(q_0\!=\!1)$, and
critical measures converge to $\mathcal{IGW}(q_0)$. 
The domain of attraction of $\mathcal{IGW}(q_0)$ for any $q_0\in[1/2,1)$ 
is characterized by the tail behavior of the offspring distribution $\{q_k\}$ of 
the initial Galton-Watson measure. 
In particular, Cor.~\ref{cor:Zipf} implies that every critical measure with
Zipf tail $q_k\sim Ck^{-(1+1/q)}$  for $q\in[1/2,1)$ and $C>0$ converges to $\mathcal{IGW}(q)$. 
The subcritical attractor $\mathcal{GW}(q_0\!=\!1)$ is the limiting
point of the IGW family for $q_0 = 1$ with generating function $Q(z)=z+(1-z) = 1$.
This distribution, however, is not prune-invariant.

Our results expand the attraction domain of the critical binary
Galton-Watson tree $\mathcal{IGW}(1/2)$ initially described by Burd et al.~\cite{BWW00}.
Specifically, Lem.~\ref{lem:2plusMoment} shows that any critical offspring 
distribution that has an infinite second moment, satisfies Assumption \ref{asm:reg},
and has a finite $2-\epsilon$ moment for all $\epsilon>0$
belongs to the attraction domain of $\mathcal{IGW}(1/2)$. 
We give an example of such a measure with $q_k\sim {4 \over 3}k^{-3}$. 

\begin{figure}[t] 
\centering\includegraphics[width=0.7\textwidth]{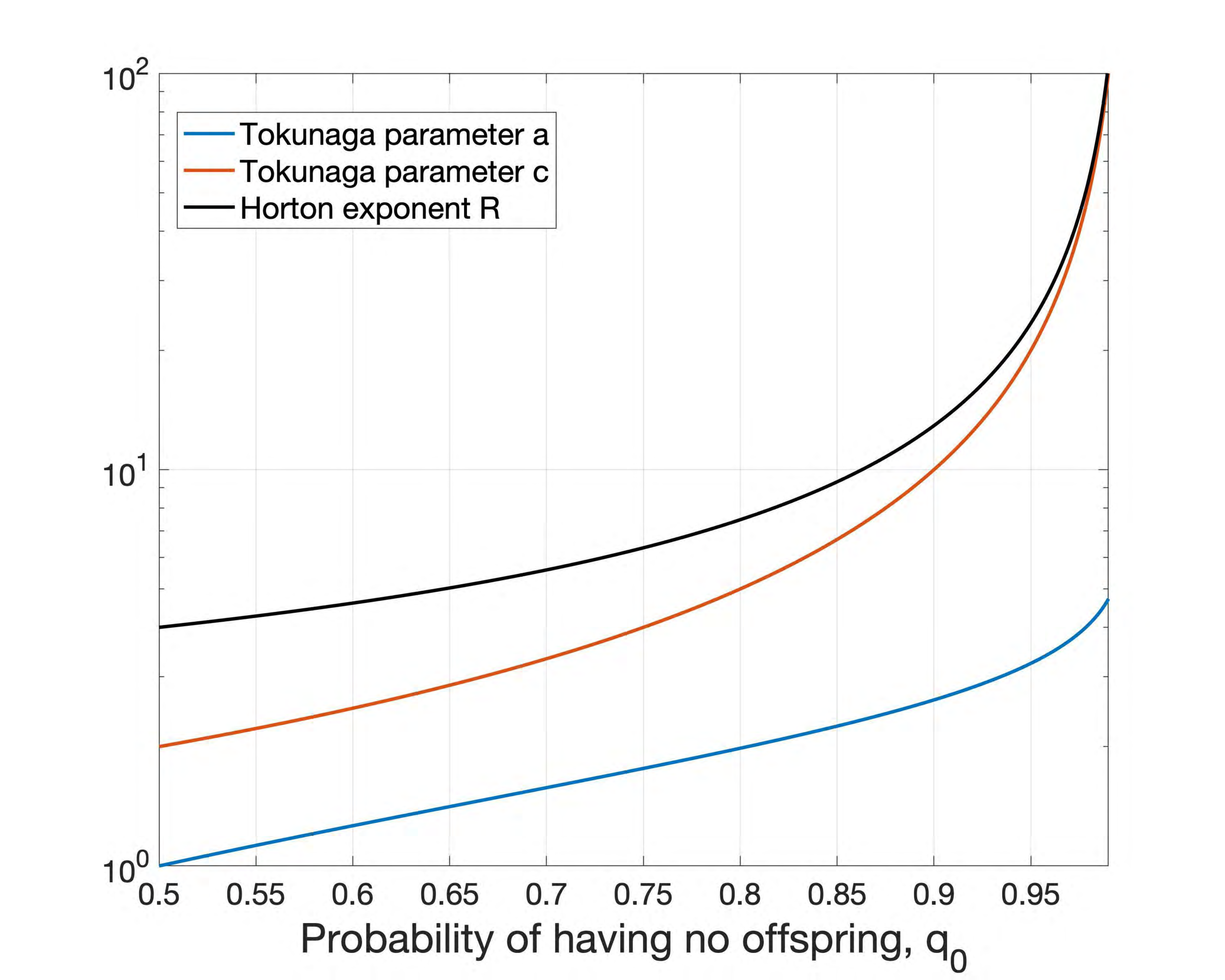}
\caption[Tokunaga parameters $a$, $c$ and Horton exponent $R$ in IGW]
{Tokunaga parameters $a$ (blue), $c$ (red) and Horton exponent $R$ (black) in 
invariant Galton-Watson trees $\mathcal{IGW}(q_0)$ for $q_0\in [0.5, 0.99]$.}
\label{fig:acR}
\end{figure}

\subsection{Toeplitz property}
The results of Burd et al.~\cite{BWW00} revealed an interesting characterization of the
critical binary Galton-Watson distribution in terms of its {\it Tokunaga sequence}.
Recall that the {\it Horton pruning} removes the leaf vertices
and their parental edges from a finite tree $T$, with subsequent
series reduction (removing degree-$2$ vertices). 
The {\it Horton order} of a tree $T$ is the minimal number of Horton prunings
sufficient to eliminate $T$.
Informally, a branch of Horton order $k$ is a contiguous part of a tree
(a collection of vertices and their parental edges in the initial tree) 
eliminated at $k$-th
iteration of Horton pruning (see Figs.~\ref{fig:HS},\ref{fig:terminology},\ref{fig:HSreg},\ref{fig:HSirreg} for examples, and Sect.~\ref{sec:HortonPruning} for a formal definition). 
Each leaf (i.e., a leaf vertex with its parental edge)  is a branch of order $1$.
Branches of higher orders may consist of lineages of vertices 
and their parental edges. 
The vertex farthest from the root is called the {\it terminal} vertex of a branch.
Applied literature often examines the statistics of mergers of branches of distinct orders within a tree.
Burd et al.~\cite{BWW00} formalize this by considering the 
{\it Tokunaga coefficients} $T_{i,j}[T]$,
for $i<j$, equal to the number of instances when a branch of order $i$
joins a non-terminal vertex of the leftmost branch of order $j$ closest to
the root within $T$, given that the tree order is greater than $j$.
This definition is suitable for describing a generic branch structure within
a Galton-Watson tree, given its symmetric iterative generation mechanism.
It has been shown \cite[Thm.~3.16]{BWW00} that the critical binary Galton-Watson
distribution $\mathcal{GW}(q_0=q_2=1/2)$ is characterized, among the bounded offspring distributions,
by the {\it Toeplitz property}:
\be\label{Toeplitz}
{\sf E}\big[T_{i,j}[T]\big] = T_{j-i}\quad\text{for a positive Tokunaga sequence } \{T_k\}_{k=1,2,\dots}.
\ee
Specifically, the critical binary Galton-Watson distribution corresponds to 
$T_k = 2^{k-1}$.
In Lem.~\ref{lem:IGWq0Tokunaga}, we show that all the invariant measures 
$\mathcal{IGW}(q_0)$ satisfy the Toeplitz property.  
In this analysis, we adopt an alternative, more general, definition
of the Tokunaga coefficient $T_{i,j}$, which (i) accounts for branching at the terminal
vertices, and (ii) can be applied to general (non Galton-Watson) trees. 
In our definition, the invariant measure $\mathcal{IGW}(q_0)$ corresponds to
the Tokunaga sequence (Lem.~\ref{lem:IGWq0Tokunaga})
\[T_1 = c^{c/(c-1)}-c-1,\quad T_k = ac^{k-1},\quad k=2,3,\dots\]
with (Fig.~\ref{fig:acR})
\[c = (1-q_0)^{-1}\quad\text{and}\quad a = (c-1)(c^{1/(c-1)}-1).\]  
The critical binary Galton-Watson case with $q_0=1/2$ corresponds to
$c=2$ and $a = 1$, which reconstructs the 
Burd et al.~\cite{BWW00} result $T_k = 2^{k-1}$.
Moreover, using the Tokunaga sequence definition from Burd et al.~\cite{BWW00},
we obtain a particularly simple Tokunaga sequence $T_k = c^{k-1}$ for $k=1,2,\dots$.

\subsection{Horton law}
A ubiquitous empirical observation in the analysis of dendritic structures
is the {\it Horton law} \cite{Horton45,KZsurvey,Pec95}. 
Informally, the law states that the numbers $N_k[T]$ of branches of order $k$ in a large
tree $T$ decays geometrically:
\[\frac{N_k[T]}{N_{k+1}[T]} \approx R\]
for some {\it Horton exponent} $R\ge 2$.
A formal definition of the Horton law for tree measures is given in Sect.~\ref{sec:Horton}.
 
It has been shown by McConnell and Gupta \cite{MG08} for a particular case of 
$T_k=ac^{k-1}$ with $a\ge0$, $c>0$, and generalized by the authors of this paper 
\cite{KZ16} to an arbitrary Tokunaga sequence $\{T_k\}$, that the 
Toeplitz property implies the Horton law. 
Lemma~\ref{lem:IGWq0Tokunaga} shows that the invariant Galton-Watson measure 
$\mathcal{IGW}(q_0)$ for any $q_0\in[1/2,1)$ obeys the Horton law 
with the Horton exponent $\,R=(1-q_0)^{-1/q_0}$ (Fig.~\ref{fig:acR}).

\begin{figure}[t] 
\centering\includegraphics[width=0.7\textwidth]{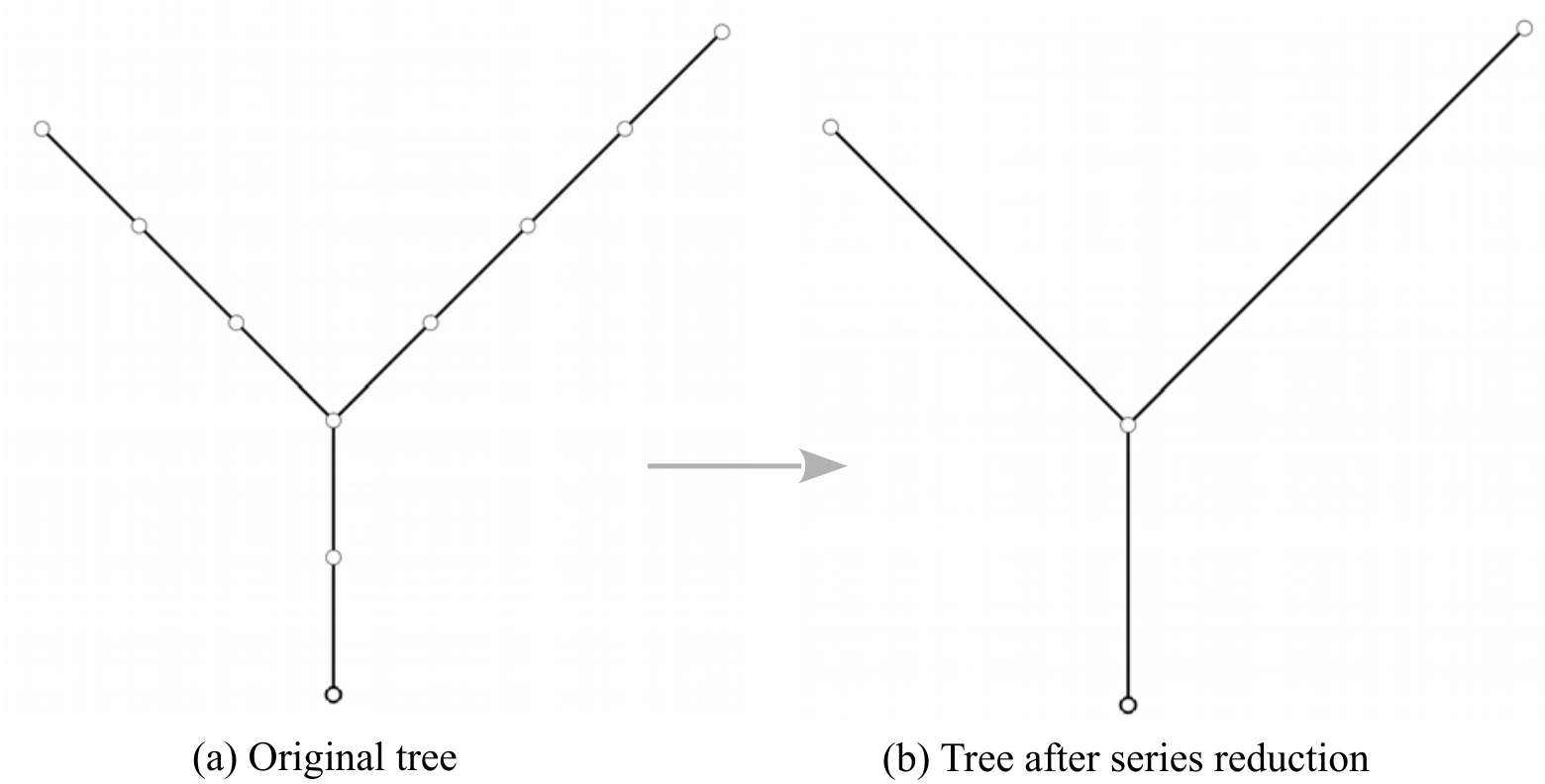}
\caption[Series reduction: Example]
{Series reduction: Example. Tree $T$ before (a) and after (b) series reduction.}
\label{fig:reduction}
\end{figure}

\section{Preliminaries}
\subsection{Galton-Watson tree measures}
Consider the space $\cT$ of finite unlabeled rooted reduced trees.
A tree is called {\it rooted} if one of its vertices, denoted by $\rho$, 
is selected as the tree root.
The existence of root imposes a parent-offspring relation between each pair of adjacent vertices: 
the one closest to the root is called the {\it parent}, 
and the other the {\it offspring}\index{vertex!parental}\index{vertex!offspring}. 
The space $\cT$ includes the {\it empty tree}\index{tree!empty} 
$\phi$ comprised of a root vertex and no edges.
The tree root is the only vertex that does not have a parent.
Let $\cT^|$ denote a subspace of {\it planted trees} in $\cT$; it contains 
$\phi$ and all the trees in $\cT$ with the root vertex having exactly one offspring 
(see Figs.~\ref{fig:reduction},\ref{fig:HS}).
The degree of the root equals the number of its offsprings. 
The degree of a non-root vertex is the number of its offsprings plus one 
(to account for the parent).
The number of the offsprings at a vertex is called the vertex {\it branching number}.
A tree from $\cT^|$ is called reduced is it has no vertices of degree $2$. 

For a given offspring distribution $\{q_k\}_{k=0,1,2,\hdots}$, we let $\mathcal{GW}(\{q_k\})$ denote the corresponding Galton-Watson tree measure. 
We assume that each tree begins with a single root vertex which produces a single offspring, so the resulting trees are in $\cT^|$. 
In this renowned Markov chain construction, each non-root vertex produces $k$ offsprings with probability $q_k$, independently of other vertices.
We assume $\sum\limits_{k=0}^\infty kq_k \leq 1$ and $q_1=0$ as we need $\mathcal{GW}(\{q_k\})$ to be a probability measure on $\cT^|$
(the trees in $\cT^|$ are required to be finite and reduced).
The assumption of subcriticality or criticality implies $q_0 \geq {1 \over 2}$,
since
\[1 \geq  \sum\limits_{k=2}^\infty kq_k \geq 2\sum_{k=2}^\infty q_k ~=2(1-q_0).\]

\subsection{Horton pruning, orders}\label{sec:HortonPruning}
Recall that {\it series reduction} on a tree $T$ removes each vertex of degree $2$ 
and merges its two adjacent edges into one (Fig.~\ref{fig:reduction}). 
\begin{Def}[{{\bf Horton pruning}}]
\label{def:Horton_prune}
Horton pruning $\cR:\cT \rightarrow \cT$ is an onto 
function whose value $\cR(T)$ for a tree $T\ne\phi$ is obtained by removing
the leaves and their parental edges from $T$, followed by series reduction.
We also set $\cR(\phi)=\phi$.
\end{Def}
\noindent
The trajectory of each tree $T$ under $\cR(\cdot)$ is uniquely determined and finite:
\be\label{TRR0}
T\equiv\cR^0(T)\to \cR^1(T) \to\dots\to\cR^k(T)=\phi,
\ee
with the empty tree $\phi$ as the (only) fixed point.
The pre-image $\cR^{-1}(T)$ of any non-empty tree $T$ consists of an infinite 
collection of trees.

\begin{figure}[t] 
\centering\includegraphics[width=0.9\textwidth]{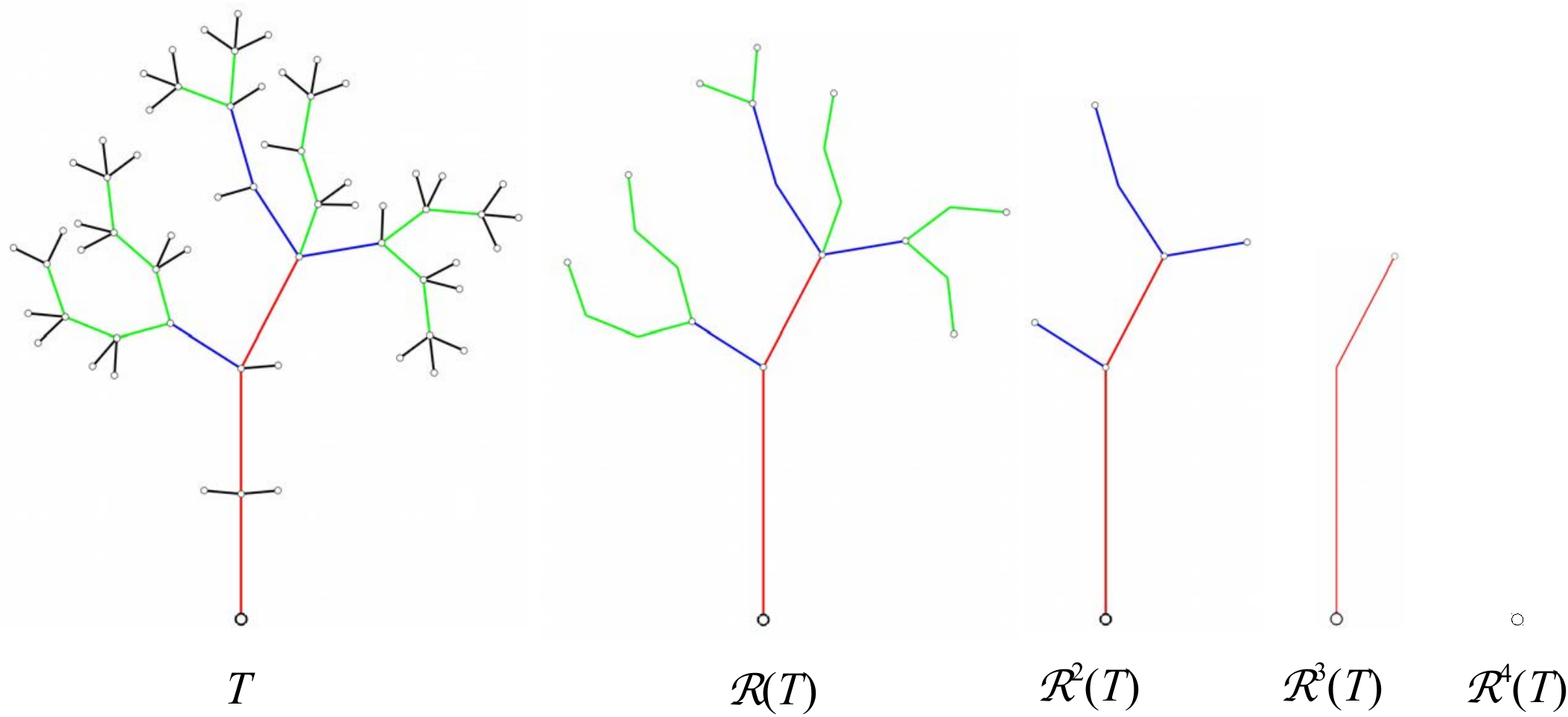}
\caption[Horton-Strahler order: Example]
{The Horton-Strahler orders: Example. Consecutive prunings $\cR^k(T)$, $k=0,1,\dots,4$, of tree $T$.
The order of tree is ${\sf ord}(T) = 4$ since $\cR^4(T)=\phi$.
Different colors depict branches of different orders: ${\sf ord}=1$ (black), ${\sf ord}=2$ (green), ${\sf ord}=3$ (blue), and ${\sf ord}=4$ (red).}
\label{fig:HS}
\end{figure}

\medskip
\noindent
It is natural to think of the distance to $\phi$ under the Horton pruning map
and introduce the respective notion of tree order \cite{Horton45,KZsurvey,Pec95,Strahler}.

\begin{Def}[{{\bf Horton-Strahler order}}]\label{def:HSorder}
The Horton-Strahler order\index{Horton-Strahler order} 
${\sf ord}(T)\in\mathbb{Z}_+$ of a tree $T\in\cT^|$  is defined as
the minimal number of Horton prunings necessary to eliminate the tree:
\[{\sf ord}(T)=\min\left\{k\ge 0 ~:~\cR^k(T)=\phi\right\}.\]
\end{Def}

\noindent 
In particular, the order of the empty tree is ${\sf ord}(\phi)=0$,
because $\cR^0(\phi)=\phi$.
This definition is illustrated in Fig.~\ref{fig:HS} for a tree $T$ with ${\sf ord}(T)=4$. 
In this paper we consider probability measures $\mathcal{GW}(\{q_k\})$ on $\cT^|$ 
that satisfy $\sum\limits_{k=0}^\infty kq_k \leq 1$ and $q_1=0$ and assign 
probability zero to the empty tree $\phi$.

\begin{figure}[t] 
\centering\includegraphics[width=0.9\textwidth]{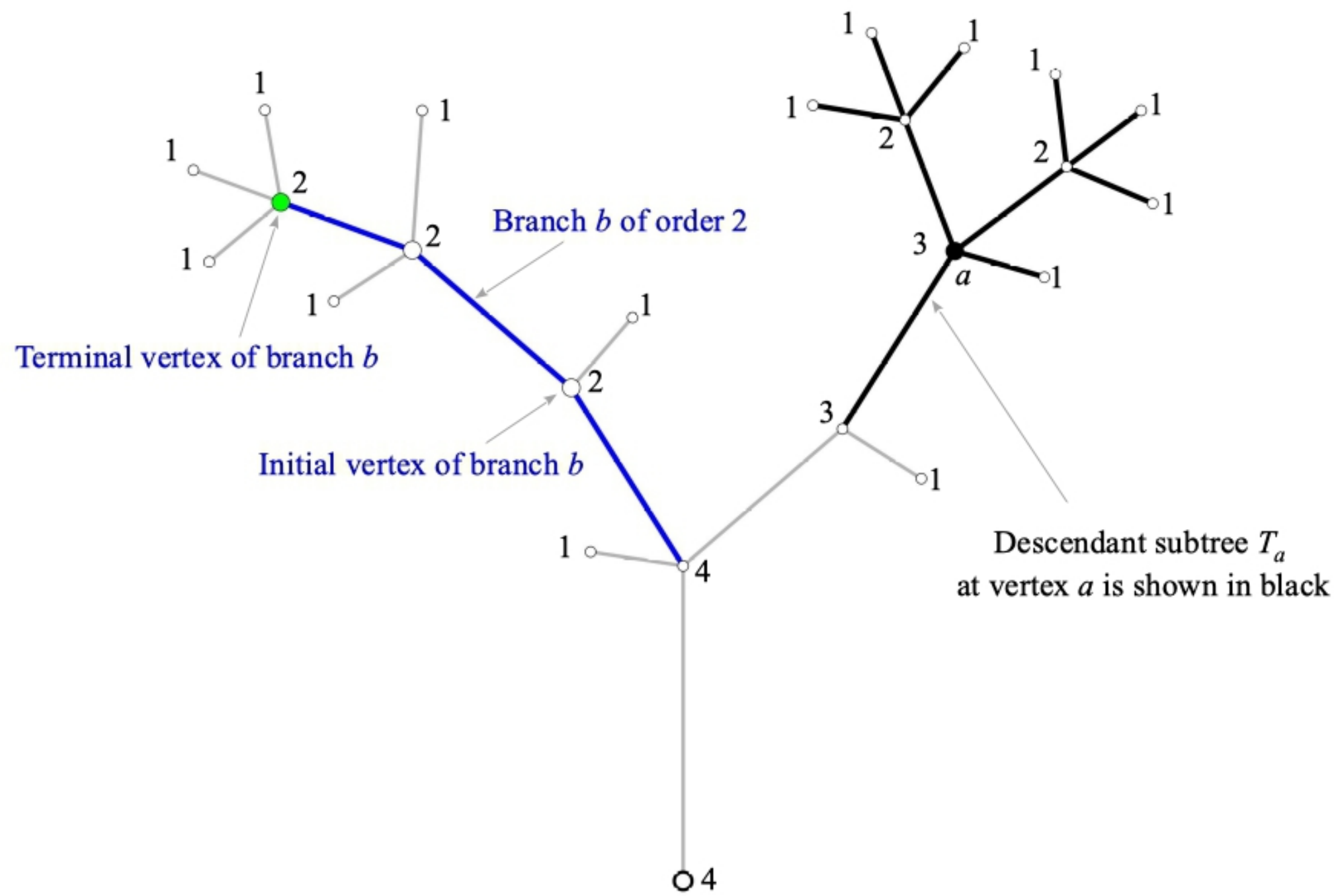}
\caption[Horton-Strahler terminology]
{Illustration of the Horton-Strahler terminology (Def.~\ref{def:HS}). 
A branch $b$ of oder $2$ is shown in blue in the left part of the figure. 
The branch consists of three vertices of order $2$ and their parental edges. 
The terminal vertex of branch $b$ is shown by green circle.
The descendant subtree $T_a$ at vertex $a$ is shown in black in the right part of the figure. 
The Horton-Strahler orders are shown next to the vertices.}
\label{fig:terminology}
\end{figure}

\medskip
\begin{Def}[{{\bf Horton-Strahler terminology}}]
\label{def:HS}
We introduce the following definitions related to the Horton-Strahler order
of a tree (see Fig.~\ref{fig:terminology}):
\begin{enumerate}
\item
{\bf (Descendant subtree at a vertex)} For any non-root vertex $v$ in $T \ne \phi$, 
a {\it descendant subtree} $T_v\subset T$ 
is the only planted subtree in $T$ rooted at the parental vertex ${\sf parent}(v)$ of $v$, and comprised by $v$ and all 
its descendant vertices together
with their parental edges.
Figure~\ref{fig:terminology} shows in black color the descendant subtree $T_a$ at vertex $a$.
\item
{\bf (Vertex order)} For any vertex $v\in T\setminus\{\rho\}$
we set ${\sf ord}(v) = {\sf ord}(T_v)$.
We also set ${\sf ord}(\rho)={\sf ord}(T)$.

\item 
{\bf (Edge order)} The parental edge of a non-root vertex has the same order as the vertex.

\item
{\bf (Branch)} A maximal connected component consisting of vertices and edges of the same order 
 is called a {\it branch}\index{branch}. Figure~\ref{fig:terminology} shows a branch $b$ of order $2$ (blue) that consists of three vertices and their parental edges.
Note that a tree $T$ always has a single branch of the maximal order ${\sf ord}(T)$.
In a stemless tree, the maximal order branch may consist of a single root vertex. 

\item
{\bf (Initial and terminal vertex of a branch)} 
The branch vertex closest to the root is called the {\it initial vertex} 
of the branch\index{branch!initial vertex}.
The branch vertex farthest from the root is called 
the {\it terminal vertex} of a branch\index{branch!terminal vertex}.
Figure~\ref{fig:terminology} shows the terminal vertex of branch $b$ (blue)
as a green circle. 
\end{enumerate} 
\end{Def}

\begin{figure}[t!] 
\centering\includegraphics[width=0.9\textwidth]{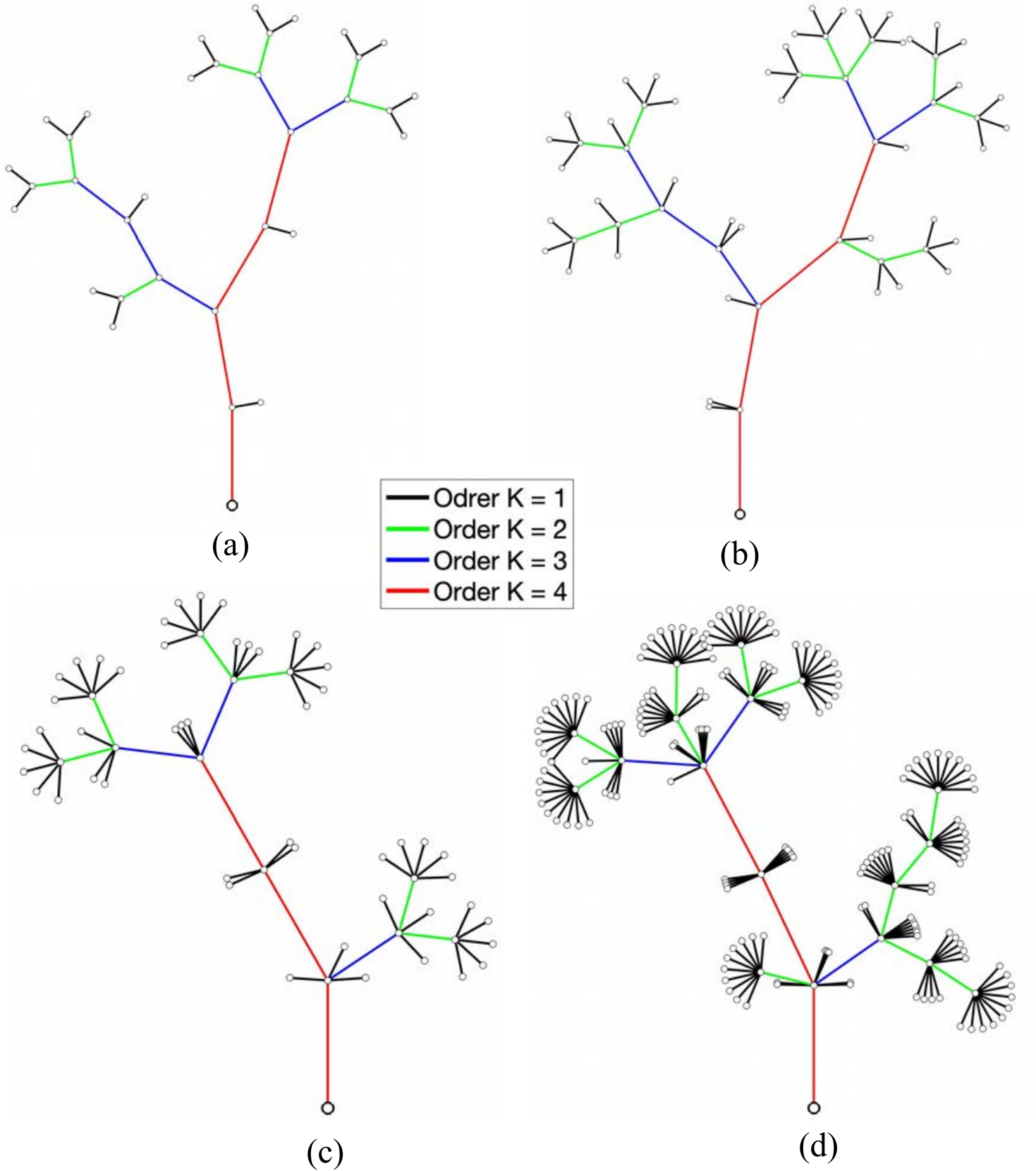}
\caption[Horton-Strahler order: Examples]
{Examples of Horton-Strahler ordering in trees with constant branching 
number $b$ ($q_0+q_b=1$).
(a) $b=2$, (b) $b=3$, (c) $b=5$, (d) $b=10$.
Each panel shows a tree of order ${\sf ord}=4$. 
Edges of different orders are shown in different colors, as
indicated in the legend.}
\label{fig:HSreg}
\end{figure}

\medskip
\noindent 
The Horton-Strahler orders can be equivalently defined via hierarchical counting \cite{Horton45,Strahler,DK94,Pec95,NTG97}.
The first such definition beyond the binary case appeared in \cite{BWW00}.
In this approach, each leaf is assigned order $1$.
If an internal vertex $p$ has $m\ge 1$ offspring with orders $i_1,i_2,\hdots,i_m$ 
and $r=\max\left\{i_1,i_2,\hdots,i_m\right\}$, then 
\be\label{eq:HSorder_gen}
{\sf ord}(p)=\begin{cases}
    r & \text{ if }~ \#\left\{s:~i_s=r \right\}=1,\\
    r +1 & \text{ otherwise}.
\end{cases}
\ee
\noindent 
The parental edge of a non-root vertex has the same order as the vertex.
The Horton-Strahler order of a tree $T\ne\phi$ is 
${\sf ord}(T)=\max\limits_{v\in T} {\sf ord}(v)$, 
where the maximum is taken over all vertices in $T$.
This definition is most convenient for practical calculations, which
explains its popularity in the literature.

Figures~\ref{fig:HSreg},\ref{fig:HSirreg} illustrate Horton-Strahler orders
in trees with a constant branching number $b$ ($q_0+q_b=1$) and 
with a bounded offspring distribution ($q_k=0$ for $k>b$), repsectively.

\begin{figure}[t!] 
\centering\includegraphics[width=0.9\textwidth]{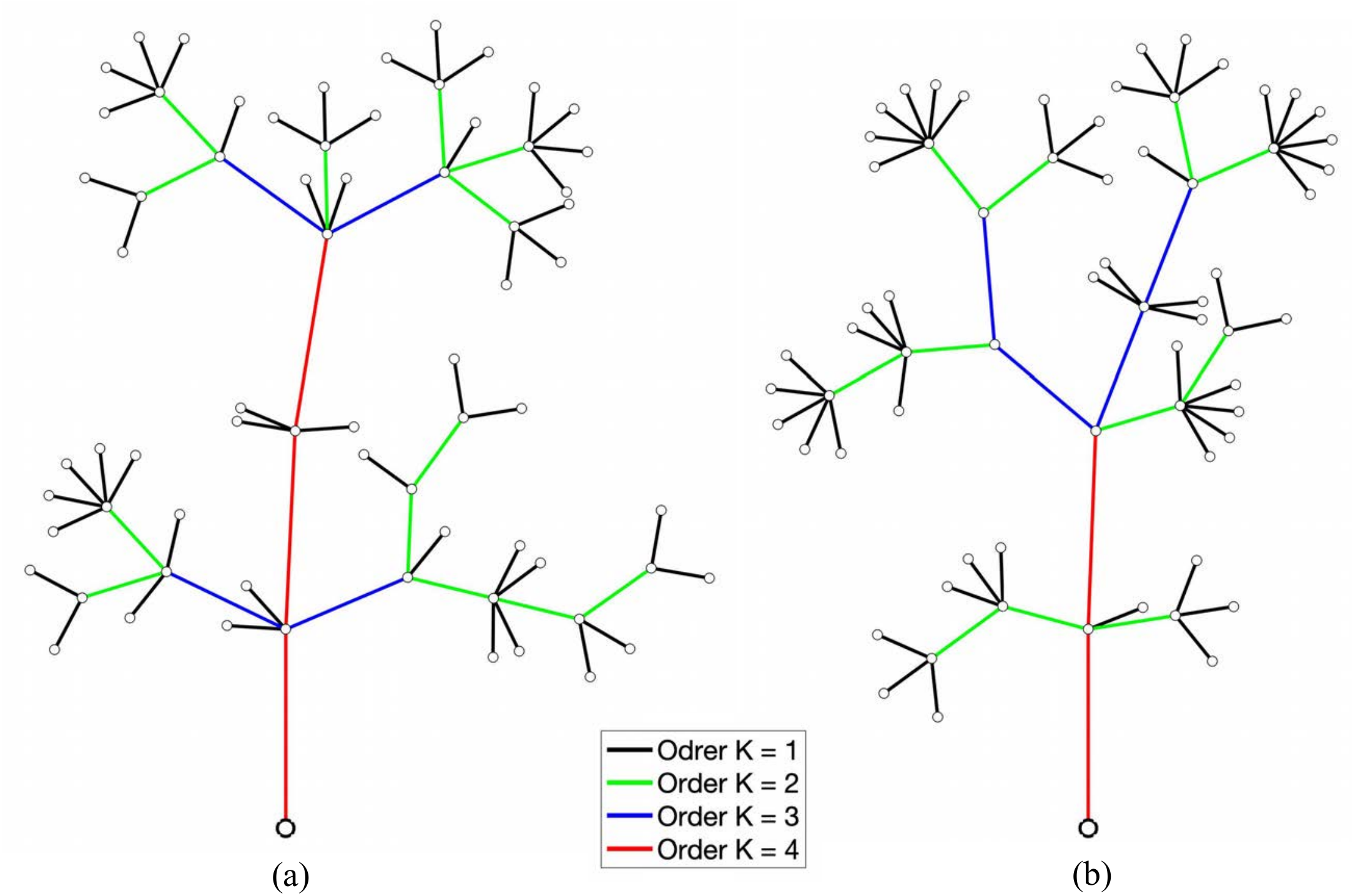}
\caption[Horton-Strahler order: Examples]
{Examples of Horton-Strahler ordering in trees with bounded offspring 
distribution: $q_k=0$ for $k>b$.
(a) $b=5$, (b) $b=6$.
Each panel shows a tree of order ${\sf ord}=4$. 
Edges of different orders are shown in different colors, as
indicated in the legend.}
\label{fig:HSirreg}
\end{figure}

\subsection{Horton self-similarity}

Here we define self-similarity of a Galton-Watson measure
with respect to the Horton pruning $\cR$, which is the main operation on trees discussed
in this work.

\begin{Def}[{{\bf Horton self-similarity}}]\label{def:prune}
Consider a Galton-Watson measure $\mu$ on $\cT$ (or $\cT^|$) such that $\mu(\phi) = 0$.
Let $\nu$ be the pushforward measure, $\nu=\cR_*(\mu)$, i.e.,
$$\nu(T)=\mu \circ \cR^{-1}(T) = \mu \big(\cR^{-1}(T)\big).$$
Measure $\mu$ is called invariant with
respect to the Horton pruning (Horton prune-invariant)\index{Horton prune-invariance},
or Horton self-similar, if for any tree $T\in\cT$ (or $\cT^|$) we have 
\be
\label{def:pi}
\nu\left(T\,|T\ne\phi\right)=\mu(T).
\ee
\end{Def}

Definition~\ref{def:prune} does not distinguish between prune-invariance
and self-similarity. 
Such equivalence is a particular
property of Galton-Watson measures connected to their Markov structure. 
In a general case, prune-invariance happens to be a weak property that allows 
a multitude of obscure measures.
A general prune-invariant measure on $\cT$ has to satisfy an additional 
property, called {\it coordination}, to be called self-similar.
The Galton-Watson measures always satisfy the coordination property; see \eqref{eq:mean_coord}.
We refer to \cite{KZsurvey} for a comprehensive 
discussion and examples.

\subsection{Tokunaga coefficients and Toeplitz property}
\label{sec:Tok}
This section introduces {\it Tokunaga coefficients} that describe
mergers of branches of different orders in a random tree. 
Empirically, a Tokunaga coefficient $T_{i,j}$ is the average number
of branches of order $i$ that merge a branch of order $j$ within 
a tree $T$.
The Markovian generation process ensures that all branches of a given order $j$ in
a Galton-Watson tree have the same probabilistic structure. 
Hence, one can follow Burd et al.~\cite{BWW00} and define $T_{i,j}$ as
the mean number of order $i$ branches within a particular branch of order $j$,
for instance -- the leftmost branch closest to the root. 
We introduce below a more general definition, which is equivalent to
that of Burd et al.~\cite{BWW00} for Galton-Watson trees, and can extend
to non Markovian branching processes.
This set up will also be needed to formulate the Horton law in Sect.~\ref{sec:Horton}.

Consider a measure $\mu$ on $\cT$ (or $\cT^|$) such that $\mu(\phi)=0$.
The Horton pruning partitions the underlying tree space into exhaustive and mutually exclusive collection 
of subspaces $\cH_k$ of trees of Horton-Strahler order $k\ge 0$ such that $\cR(\cH_{k+1})=\cH_k$.
Here $\cH_0=\{\phi\}$, $\cH_1$ consists of a single tree comprised of a root and a leaf
connected to the root by its parental edge, and all other subspaces $\cH_k$, $k\ge 2$, consist of 
an infinite number of trees.
Naturally, $\cH_k\bigcap\cH_{k'}=\emptyset$ if $k\ne k'$, and 
$\bigcup\limits_{k\ge 1} \cH_k =\cT$ (or $\cT^|$).
Consider a set of conditional probability measures $\{\mu_k\}_{k\ge 0}$
each of which is defined on $\cH_k$.
Specifically, we set $\mu_k(\cdot)\equiv0$ for any $k$ such that $\mu(\cH_k)=0$
and  
\be\label{eq:muk}
\mu_k(T) = \mu(T\,|T\in\cH_k)
\ee
otherwise. 
Letting $\pi_k=\mu(\cH_k)$, the measure $\mu$ can be represented as a mixture of the conditional measures:
\be
\label{muk}
\mu=\sum_{k=1}^{\infty}\pi_k\mu_k.
\ee

Let $N_k=N_k[T]$ be the number of branches of order $k$ in a tree $T$. 
For given integers $1 \leq i <j$, let $n_{i,j}=n_{i,j}[T]$ denote the total number 
of vertices of order $i$ that have parent of order $j$ 
in a tree $T \in \cT$ (or $\cT^|$).
We write ${\sf E}_K[\cdot]$ for the expectation with respect to $\mu_{K}$ of Eq.~\eqref{eq:muk}.

\medskip
\noindent
We define the {\it average Horton numbers}\index{Horton numbers (branch counts)} 
for subspace $\cH_{K}$ as
\[\cN_{k}[K]= {\sf E}_K[N_k],\quad 1\le k\le K,\quad K\ge 1.\]
For subspace $\cH_{K}$, let
\be
\label{eqn:tijGeneral}
t_{i,j}[K]=\frac{{\sf E}_K[n_{i,j}]}{{\sf E}_K[N_j]}=\frac{{\sf E}_K[n_{i,j}]}{\cN_{j}[K]}, \quad 1\le i <j\le K,
\ee
be the {\it total Tokunaga merger statistics} that is used
to define the {\it Tokunaga coefficients} 
\be\label{eqn:GeneralTij}
T_{i,j}[K]=t_{i,j}[K]-2\delta_{i,j-1} \qquad (1 \leq i <j).
\ee 
\begin{rem}\label{rem:side-branching}
Recall that a branch of order $j$ is formed by a merger of two or more branches of 
order $j-1$.
We designate two arbitrarily selected branches of order $j-1$ that descend 
from the terminal vertex of a branch of order $j$ as {\it principle branches}. 
The existence of such two branches follows from the definition of 
the Horton order (Def.~\ref{def:HSorder}). 
The other branches (if any) of order $i \leq j-1$ that descend from any 
vertex, including the terminal vertex, in a branch of order $j$ are
said to be {\it side branches} of {\it Tokunaga index} $\{i,j\}$.
The Tokunaga coefficients $T_{i,j}[K]$ are intended to count the number of side branches of Tokunaga index $\{i,j\}$, which explains the need to subtract $2\delta_{i,j-1}$ in \eqref{eqn:GeneralTij}.
\end{rem}

Finally, let $n_{i,j}^o$ denote the total number of vertices of order $i$ whose 
parent vertices are non-terminal vertices of order $j$. Then,
\be\label{eqn:GeneralTijReg}
T_{i,j}^o[K]={{\sf E}_K[n_{i,j}^o] \over {\sf E}_K[N_j]}={{\sf E}_K[n_{i,j}^o] \over \cN_{j}[K]} \qquad (1 \leq i <j)
\ee
are called the {\it regular Tokunaga coefficients}.

\begin{rem}
The quantities $\cN_k[K]$, $t_{i,j}[K]$, $T_{i,j}[K]$, and $T_{i,j}^o[K]$ depend on the measure $\mu$.
We skip this dependence in our notations.
\end{rem}

\noindent
We observe that for a subcritical or critical Galton-Watson measure $\mu$ we have 
the following {\it coordination} property \cite{KZsurvey}:
\be\label{eq:mean_coord}
T_{i,j}:=T_{i,j}[K] \quad \text{ for all } K\ge 2 \text{ and } 1\le i< j\le K.
\ee
This is explained as follows. 
Consider all nodes in generation $d \in \mathbb{N}$ (which may be an empty set) in
a critical or subcritical Galton-Watson tree $T$. 
The descendant subtrees $T_v$ (see Def.~\ref{def:HS}) for $v$ in  generation $d$ are independently distributed according to $\mu$.
Sampling of $T_v$ can be split into two steps, first selecting its order with probability distribution $\pi_j$, next sampling
the tree of order $j$ according to the probability measure $\mu_j$. 
The branching history $\mathcal{F}_d$ up to generation $d$ together with 
the orders of the descendant subtrees $T_v$ with $v$ in generation $d$ completely 
determines (i) the order of the tree $T$, and 
(ii) whether or not $v$ is the initial vertex (Def.~\ref{def:HS}) of the corresponding branch of order ${\sf ord}(T_v)$.
At the same time, conditioned on $\mathcal{F}_d$ and the orders ${\sf ord}(T_v)$ for $v$ in generation $d$, each $T_v$ is independently distributed
according to $\mu_j$, where ${\sf ord}(T_v)=j$.

\medskip
\noindent
The respective Tokunaga matrix $\mathbb{T}_K$ is a $K \times K$ matrix
\[\mathbb{T}_K=\left[\begin{array}{ccccc}
0 & T_{1,2} & T_{1,3} & \hdots & T_{1,K} \\
0 & 0 & T_{2,3} & \hdots & T_{2,K} \\
0 & 0 & \ddots & \ddots & \vdots \\
\vdots & \vdots & \ddots & 0 & T_{K-1,K} \\
0 & 0 & \dots & 0 & 0\end{array}\right],\]
which coincides with the restriction of any larger-order Tokunaga matrix $\mathbb{T}_M$, $M>K$,
to the first $K\times K$ entries.

\begin{Def}[{{\bf Toeplitz property}}]
\label{Tsi}
A Galton-Watson measure $\mu$ is said to satisfy 
the {\it Toeplitz property} if 
there exists a sequence $T_k\ge 0$, $k=1,2,\dots$
such that
\be\label{eq:Toeplitz}
T_{i,j} = T_{j-i}.
\ee
The elements of the sequences $T_k$ are also referred to as Tokunaga coefficients, 
which does not create confusion with $T_{i,j}$.
\end{Def}

For a Galton-Watson measure that satisfies the Toeplitz property,
the corresponding Tokunaga matrices $\mathbb{T}_K$ are Toeplitz\footnote{Note that in \cite{BWW00}, the Tokunaga sequence was set to satisfy $T_{i,j}^o=T_{i,j}^o[K]=T_{j-i}$. That is, the offsprings adjacent to the terminal vertex of order $j$ branch were not counted.}:
$$\mathbb{T}_K=\left[\begin{array}{ccccc}
0 & T_1 & T_2 & \hdots & T_{K-1} \\
0 & 0 & T_1 & \hdots & T_{K-2} \\
0 & 0 & \ddots & \ddots & \vdots \\
\vdots & \vdots & \ddots & 0 & T_1 \\
0 & 0 & \dots & 0 & 0\end{array}\right].$$

The following statement has been proven in \cite{KZsurvey} for 
(not necessarily Galton-Watson) binary trees;
the argument applies verbatim to general Galton-Watson trees. 
\begin{prop}[{{\bf Prune-invariance implies Toeplitz}}]
\label{PItoT}
Suppose a Galton-Watson measure $\mu$ is Horton prune-invariant, then
it satisfies the Toeplitz property (Def.~\ref{Tsi}).
\end{prop}

\begin{Def}[{{\bf Tokunaga self-similarity}}]
\label{ss2}
A Galton-Watson measure $\mu$ on $\cT$ is called 
{\it Tokunaga self-similar} with parameters $(a,c)$ if it satisfies the
Toeplitz property and its
Tokunaga sequence $\{T_j\}_{j=1,2,\hdots}$ is expressed as
\be \label{eq:tok}
T_j = a\,c^{j-1},\quad j\ge 1
\ee
for some constants $a\ge0$ and $c>0$.
\end{Def}

\subsection{Horton law}
\label{sec:Horton}
Consider a measure $\mu$ on $\cT$ (or $\cT^|$) and its conditional measures $\mu_K$, each defined
on subspace $\cH_K\subset\cT$ of trees of Horton-Strahler 
order $K\ge 1$ as discussed in Sect.~\ref{sec:Tok}. 
We write $T\stackrel{d}{\sim}\mu_K$ for a random tree $T$ drawn from a subspace 
$\cH_K$ ($\mu(\cH_K)>0$) according to measure~$\mu_K$.
\begin{Def}[{{\bf Strong Horton law for mean branch numbers}}] 
\label{def:Horton_mean}
We say that a probability measure $\mu$ on $\cT$ (or $\cT^|$) satisfies a strong Horton law for mean branch numbers
if there exists such a positive (constant) Horton exponent $R\ge2$ that for any $k\ge 1$, 
\be
\label{eq:Horton_mean}
\lim_{K\to\infty}\frac{\cN_k[K]}{\cN_1[K]}
= R^{1-k}.
\ee
\end{Def}
\noindent Here, the adjective {\it strong} refers to the type of geometric convergence; see \cite{KZsurvey}
for details.

The work \cite{KZ16} establishes the strong Horton law in a binary tree
that satisfies the Toeplitz property (Def.~\ref{Tsi}).
We observe that the results of \cite{KZ16} hold beyond the binary case,
as the derivation steps are identical.
Specifically, assume the Toeplitz property with a Tokunaga sequence $\{T_k\}$ 
and consider a sequence $t(k)$ defined by 
\[t(0)=-1, ~\text{ and } ~~t(k)=T_k+2\delta_{1,k} ~\text{ for } k \geq 1.\] 
Observe that $t_{i,j}=t_{i,j}[K]=t(j-i)$. The generating function of $t(k)$ is
$$\hat{t}(z)=\sum\limits_{k=0}^\infty z^k t(k)=-1+2z+\sum\limits_{k=1}^\infty z^k T_k.$$

\begin{thm}[{{\bf Strong Horton law in a mean self-similar tree, \cite{KZ16}}}]
\label{thm:HLSST}
Suppose $\mu$ is a Galton-Watson measure on $\cT^|$ that satisfies the Toeplitz
property with Tokunaga sequence $\{T_j\}_{j=1,2,\hdots}$ such that
\be
\label{eq:tamed}
\limsup_{j\to\infty} T_j^{1/j}<\infty.
\ee
Then the strong Horton law for mean branch numbers (Def.~\ref{def:Horton_mean}) holds
with the Horton exponent $R=1/w_0$, where $w_0$ is the only real zero of
the generating function $\hat{t}(z)$ in the interval $\left(0,{1 \over 2}\right]$. 
Moreover,
\be\label{eq:Nk}
\lim_{K\to\infty}\left(\cN_1[K]\,R^{-K}\right) = const. > 0.
\ee
Conversely, if 
$\limsup\limits_{j \rightarrow \infty} T_j^{1/j} =\infty$, 
then the limit 
$\lim\limits_{K\to\infty}\frac{\cN_{k}[K]}{\cN_1[K]}$ 
does not exist at least for some $k$.
\end{thm}

\section{Main results}

\subsection{Distribution of Horton orders and related functions}
Consider a collection of critical or subcritical Galton-Watson measures $\mathcal{GW}(\{q_k\})$ with $q_1=0$ on $\cT^|$.
Let $Q(z) =\sum\limits_{m=0}^\infty z^m q_m$ for $z \in [0,1]$ be the generating function of $\{q_k\}$. 
For $T\stackrel{d}{\sim}\mathcal{GW}(\{q_k\})$ we denote $\pi_j :=P\big({\sf ord}(T)=j\big)$. Finally, let $\sigma_0=0$ and $\sigma_j:=\sum\limits_{i=1}^j \pi_i$ ($j \geq 1$). 

\begin{lem}[{\bf Order distribution}]\label{lem:GWtreePIj}
Consider a Galton-Watson measure $\mathcal{GW}(\{q_k\})$ with $q_1=0$.
Assume criticality or subcriticality, i.e., $\sum\limits_{k=0}^\infty kq_k \leq 1$. Then,
\be\label{eqn:GWtreePIj}
\pi_1=q_0~~\text{ and }~~\pi_j={Q(\sigma_{j-1})-Q(\sigma_{j-2})-\pi_{j-1}Q'(\sigma_{j-2}) \over 1-Q'(\sigma_{j-1})}
~~(j \geq 2).
\ee
\end{lem}
\begin{proof}
The probability of tree $T$ with a single leaf is $\pi_1=P\big({\sf ord}(T)=1\big)=q_0$.

Next we find the Horton-Strahler order of the offspring of the root using the rule \eqref{eq:HSorder_gen}.
The probability that the offspring of the root is a terminal vertex of a branch of order $j$, $j \geq 2$, is 
$$\sum\limits_{m=2}^\infty q_m \sum\limits_{\ell=2}^m \binom{m}{\ell} \pi_{j-1}^\ell \sigma_{j-2}^{m-\ell}=Q(\sigma_{j-1})-Q(\sigma_{j-2})-\pi_{j-1}Q'(\sigma_{j-2}).$$
Here we take a sum over all possible numbers $m\ge 2$ of offsprings, 
and calculate the probability that $\ell\ge 2$ of the offsprings have order $j-1$,
while the other $m-\ell$ offsprings have orders less than $j-1$.

Similarly, the probability of the offspring of the root to be a regular (non-terminal) 
vertex of order $j$ equals 
\be\label{eqn:RegVertexQp}
\sum\limits_{m=2}^\infty q_m m\pi_j \sigma_{j-1}^{m-1}=\pi_jQ'(\sigma_{j-1}).
\ee 
Therefore, 
$$\pi_j =\pi_jQ'(\sigma_{j-1})+\Big(Q(\sigma_{j-1})-Q(\sigma_{j-2})-\pi_{j-1}Q'(\sigma_{j-2})\Big),$$
which implies \eqref{eqn:GWtreePIj}.
\end{proof}

\begin{figure}[t] 
\centering\includegraphics[width=0.75\textwidth]{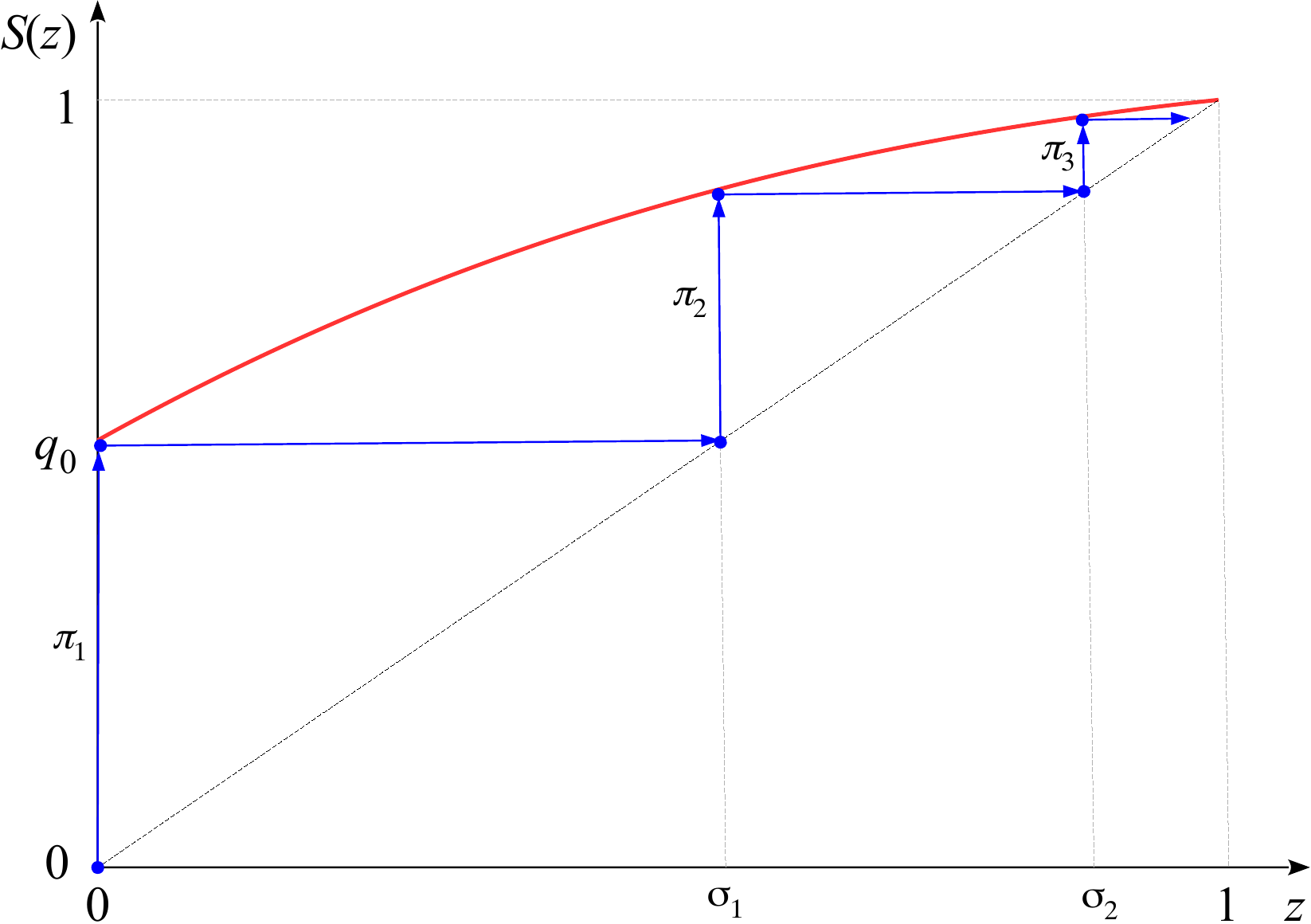}
\caption[Illustration to Cor.~\ref{cor:GWtreePIj}]
{Illustration to Cor.~\ref{cor:GWtreePIj}.
Function $S(z)$ is shown in red. 
Equation~\eqref{eqn:GWtreeSigmaj} implies that the values of $\sigma_j$
are obtained by iterative application of $S(t)$, starting with $\sigma_0=0$. 
These iterations are illustrated by blue lines with arrows.
Vertical increments correspond to the values of $\pi_j$.}
\label{fig:rec}
\end{figure}

\begin{cor}\label{cor:GWtreePIj}
Consider a Galton-Watson measure $\mathcal{GW}(\{q_k\})$ with $q_1=0$.
Assume criticality or subcriticality, i.e., $\sum\limits_{k=0}^\infty kq_k \leq 1$. 
Then, $\sigma_j$ can be expressed via an iterated function (Fig.~\ref{fig:rec}) 
\be\label{eqn:GWtreeSigmaj}
\sigma_j=\underbrace{S \circ  \hdots \circ S}_{j \text{ times}}(0) \qquad \text{ for }~j \geq 1,
\ee
where
\be\label{eqn:GWtreeSx}
S(z)={Q(z)-zQ'(z) \over 1-Q'(z)}.
\ee
\end{cor}
\begin{proof}
Equation \eqref{eqn:GWtreePIj} implies
\be\label{eqn:iterQandPIj}
\pi_j=\big[Q(\sigma_{j-1})+\pi_jQ'(\sigma_{j-1})\big]-\big[Q(\sigma_{j-2})+\pi_{j-1}Q'(\sigma_{j-2})\big] \qquad \text{ for }~j \geq 2.
\ee
Hence, summing up the terms in \eqref{eqn:iterQandPIj}, and substituting $\pi_1=q_0$, we obtain
$$\sigma_j=\sum\limits_{i=1}^j \pi_i=Q(\sigma_{j-1})+\pi_jQ'(\sigma_{j-1})=Q(\sigma_{j-1})+(\sigma_j-\sigma_{j-1})Q'(\sigma_{j-1})$$
for all $j \geq 1$. Thus, $\sigma_j={Q(\sigma_{j-1})-\sigma_{j-1}Q'(\sigma_{j-1}) \over 1-Q'(\sigma_{j-1})}=S(\sigma_{j-1})$.
\end{proof}

\medskip
\noindent
Set $S(1)=\lim\limits_{x \rightarrow 1-} {Q(x)-xQ'(x) \over 1-Q'(x)}$. 
Then, by L'H\^{o}pital's rule,
$\,S(1)=\lim\limits_{x \rightarrow 1-}{xQ''(x) \over Q''(x)}=1$.
Next, for the progeny variable $X \stackrel{d}{\sim} \{q_k\}$, consider the following important function
\be\label{eqn:gxDef}
g(x)=\sum\limits_{m=0}^\infty {\sf E}\big[(X-m-1)_+\big] \, x^m ~=\sum\limits_{m=0}^\infty \sum\limits_{k=m+1}^\infty (k-m-1)q_k \,x^m,
\ee
where $x_+=\max\{x,0\}$.

\begin{prop}\label{prop:Qminusx}
For a critical (i.e., $Q'(1)=1$) Galton-Watson process $\mathcal{GW}(\{q_k\})$ with $q_1=0$, we have
$$Q(x)-x=(1-x)^2 g(x)$$
for $g(x)$ as defined in \eqref{eqn:gxDef}.
\end{prop}
\begin{proof}
Since $\sum\limits_{k=2}^\infty kq_k=Q'(1)=1$, 
\begin{align}\label{eqn:QminusxCrit}
Q(x)-x&=q_0+\sum\limits_{k=2}^\infty q_k x^k ~-q_0x-\sum\limits_{k=2}^\infty q_kx 
~~=(1-x)\left[q_0-\sum\limits_{k=2}^\infty q_k {1 -x^{k-1} \over 1-x}x\right] \nonumber \\
&=(1-x)\left[\sum\limits_{k=2}^\infty kq_k-\sum\limits_{k=2}^\infty q_k-\sum\limits_{k=2}^\infty q_k {1 -x^{k-1} \over 1-x}x\right] 
=(1-x)\sum\limits_{k=2}^\infty q_k \left(k-1-\sum\limits_{j=1}^{k-1} x^j \right) \nonumber \\
&=(1-x)\sum\limits_{k=2}^\infty q_k \left(\sum\limits_{j=1}^{k-1} (1-x^j) \right)
=(1-x)^2 \sum\limits_{k=2}^\infty q_k \sum\limits_{j=1}^{k-1} \sum\limits_{m=0}^{j-1}x^m \nonumber \\
&=(1-x)^2 \sum\limits_{k=2}^\infty q_k \sum\limits_{m=0}^{k-2} (k-m-1)x^m 
=(1-x)^2  \sum\limits_{m=0}^\infty \sum\limits_{k=m+2}^\infty (k-m-1)q_k x^m \nonumber \\
&=(1-x)^2 g(x).
\end{align}
\end{proof}

\medskip
\noindent
Let $L$ denote the limit $\,\lim\limits_{x \rightarrow 1-}\left({\ln{g(x)} \over -\ln(1-x)}\right)\,$ whenever the limit exists.

\medskip

\begin{lem}\label{lem:2plusMoment}
For the progeny variable $X \stackrel{d}{\sim} \{q_k\}$ and $g(x)$ as defined in \eqref{eqn:gxDef}, if
\be\label{eqn:2plusMoment}
{\sf E}[X^{2-\epsilon}]=\sum \limits_{k=0}^\infty k^{2-\epsilon}q_k ~<\infty \qquad \forall \epsilon>0,
\ee
then $~L=\lim\limits_{x \rightarrow 1-}\left({\ln{g(x)} \over -\ln(1-x)}\right)=0$.
\end{lem}
\begin{proof}
Suppose \eqref{eqn:2plusMoment} holds, then by the Dominated Convergence Theorem,
as $m \rightarrow \infty$,
\be\label{eqn:Xmplus1}
(m+1)^{1-\epsilon} \,{\sf E}\big[(X-m-1)_+\big] \leq {\sf E}\big[X^{1-\epsilon}(X-m-1)_+\big]
\leq {\sf E}\big[X^{2-\epsilon} \,{\bf 1}_{\{X \geq m+1\}}\big] ~\rightarrow 0.
\ee
Accordingly, the $m$-th coefficient ${\sf E}\big[(X-m-1)_+\big]$ in the 
power series representation \eqref{eqn:gxDef} of $g(x)$ is 
$o\left(m^{\epsilon-1}\right)$.
Next, for $\epsilon>0$, the $m$-th coefficient of the power series 
expansion of $(1-x)^{-\epsilon}$ is
\be\label{eqn:1minusX}
{\prod\limits_{i=0}^{m-1}(\epsilon+i) \over m!}={\Gamma(\epsilon+m) \over \Gamma(\epsilon)m!} \sim \frac{m^{\epsilon-1}}{\Gamma(\epsilon)},\quad m\to\infty.
\ee
Together, \eqref{eqn:Xmplus1} and \eqref{eqn:1minusX} imply
\[\limsup\limits_{x \rightarrow 1-}{\ln{g(x)} \over \ln(1-x)^{-\epsilon}} \leq 1
\quad \Leftrightarrow \quad
\limsup\limits_{x \rightarrow 1-}{\ln{g(x)} \over -\ln(1-x)} \leq \epsilon \qquad \forall \epsilon>0.\]
Hence,
$~\limsup\limits_{x \rightarrow 1-}{\ln{g(x)} \over -\ln(1-x)} = 0$,
while obviously
$\,\liminf\limits_{x \rightarrow 1-}{\ln{g(x)} \over -\ln(1-x)} \geq 0$.
\end{proof}

\subsection{Regularity condition}
\label{sec:reg}
Many of the results of the paper are proved under the following assumption.
\begin{assum}\label{asm:reg}
The following limit exists:
\be\label{eqn:RegAsm}
S'(1)=\lim\limits_{x \rightarrow 1-}{1-S(x) \over 1-x}.
\ee
\end{assum}

\noindent
Observe that since $\,S(x)-x={Q(x)-x \over 1-Q'(x)}$,
Assumption \ref{asm:reg} is equivalent to the existence of the limit
\be\label{eqn:RegAsmQ}
\lim\limits_{x \rightarrow 1-}{Q(x)-x \over (1-x)\big(1-Q'(x)\big)} = 1-S'(1).
\ee

\medskip
\begin{lem}\label{lem:dSof1}
Consider a critical Galton-Watson measure $\mathcal{GW}(\{q_k\})$ with $q_1=0$.
If Assumption \ref{asm:reg} is satisfied, then for $g(x)$ defined in \eqref{eqn:gxDef} 
the following limit exists
\be\label{eqn:gLnLimL}
\lim\limits_{x \rightarrow 1-}\left({\ln{g(x)} \over -\ln(1-x)}\right)=L,
\ee
and
$\, S'(1)={1-L \over 2-L}$.
\end{lem}
\begin{proof}
By the L'H\^{o}pital's rule,
\begin{align*}\label{eqn:dSof1}
L&=\lim\limits_{x \rightarrow 1-}\left({\ln{g(x)} \over -\ln(1-x)}\right)=2-\lim\limits_{x \rightarrow 1-}{\ln{g(x)}+2\ln(1-x) \over \ln(1-x)}
=2-\lim\limits_{x \rightarrow 1-}{\ln\big(Q(x)-x\big) \over \ln(1-x)} \nonumber \\
&=2-\lim\limits_{x \rightarrow 1-}{{d \over dx}\ln\big(Q(x)-x\big) \over {d \over dx}\ln(1-x)}=2-\lim\limits_{x \rightarrow 1-}{(1-x)\big(1-Q'(x)\big) \over Q(x)-x}
=2-\lim\limits_{x \rightarrow 1-}{1-x \over S(x)-x} \nonumber \\
&=2-{1 \over 1-S'(1)}.
\end{align*}
\end{proof}

\medskip
\begin{rem}\label{rem:gapLHrule}
Notice that due to the conditions for the L'H\^{o}pital's rule, there are cases when the limit in \eqref{eqn:gLnLimL} exists 
while the limit in \eqref{eqn:RegAsm} does not exists. Indeed, the L'H\^{o}pital's rule in  \eqref{eqn:gLnLimL} holds under the condition
that the limit $\lim\limits_{x \rightarrow 1-}{{d \over dx}\ln(Q(x)-x) \over {d \over dx}\ln(1-x)}$ exists, or diverges to infinity.
This gap will be apparent in the context of Lem.~\ref{lem:LimitExists} and
Thm.~\ref{thm:completeGW}.
\end{rem}

\medskip
\begin{rem}\label{rem:subcritical}
Assumption \ref{asm:reg} is satisfied with $S'(1)=0$ for a subcritical Galton-Watson process $\mathcal{GW}(\{q_k\})$ with $q_1=0$. Indeed, we have
\begin{align*}
\lim\limits_{x \rightarrow 1-}{Q(x)-x \over 1-x} &=\lim\limits_{x \rightarrow 1-}\left(q_0-\sum\limits_{k=2}^\infty q_k {1 -x^{k-1} \over 1-x}x\right) 
=\lim\limits_{x \rightarrow 1-}\left(q_0-\sum\limits_{k=2}^\infty (k-1)q_k \right) \nonumber \\
&=1-\sum\limits_{k=2}^\infty kq_k =1-Q'(1) >0
\end{align*}
and, therefore,
\be\label{eqn:dSof1subcrit}
S'(1)=1-\lim\limits_{x \rightarrow 1-}{Q(x)-x \over (1-x)\big(1-Q'(x)\big)}=1-{1-Q'(1) \over 1-Q'(1)}=0.
\ee
\end{rem}

\medskip
\begin{lem}\label{lem:2ndMomentAsm}
Consider a critical Galton-Watson measure $\mathcal{GW}(\{q_k\})$ with $q_1=0$.
If the second moment of the offspring distribution is finite,
$${\sf E}[X^2]=\sum \limits_{k=0}^\infty k^2q_k ~<\infty,$$
then Assumption \ref{asm:reg} is satisfied with $\,S'(1)={1 \over 2}$ and 
$L=\lim\limits_{x \rightarrow 1-}\left({\ln{g(x)} \over -\ln(1-x)}\right)=0$.
\end{lem}
\begin{proof}
By L'H\^{o}pital's rule,
$$\lim\limits_{x \rightarrow 1-}{Q(x)-x \over (1-x)^2}={1 \over 2} \lim\limits_{x \rightarrow 1-}{1-Q'(x) \over 1-x}={Q''(1) \over 2}.$$
Thus,
$$S'(1)=1-\lim\limits_{x \rightarrow 1-}\left({Q(x)-x \over (1-x)^2}{1-x \over 1-Q'(x)}\right)=1-{Q''(1) \over 2Q''(1)}={1 \over 2},$$
and by Lem.~\ref{lem:dSof1}, $\,L=2-{1 \over 1-S'(1)}=0$.
\end{proof}

\medskip
\noindent
The next statement suggests a sufficient condition for Assumption \ref{asm:reg}.
\begin{lem}[{{\bf Regularity condition}}]\label{lem:RegCond}
Consider a critical Galton-Watson process $\mathcal{GW}(\{q_k\})$ with $q_1=0$ and infinite second moment, 
i.e., $\sum\limits_{k=0}^\infty k^2 q_k =\infty$.
Suppose that for the progeny variable $X \stackrel{d}{\sim} \{q_k\}$ the following
limit exists:
\be\label{eqn:RegCondLim}
\Lambda=\lim\limits_{k \rightarrow \infty}{k \over E[X\,|\,X\geq k]}=\lim\limits_{k \rightarrow 
\infty}{~k\sum\limits_{m=k}^\infty \!q_m ~ \over \sum\limits_{m=k}^\infty \! mq_m}.
\ee 
Then Assumption \ref{asm:reg} is satisfied with $S'(1)=\Lambda$.
\end{lem}
\begin{proof}
For $x \in (0,1)$,
$${1-Q'(x) \over 1-x}={Q'(1)-Q'(x) \over 1-x}=\sum\limits_{m=0}^\infty \!mq_m {1-x^{m-1} \over 1-x}
=\sum\limits_{m=0}^\infty \sum\limits_{k=0}^{m-2} \!mq_m x^k=\sum\limits_{k=0}^\infty b_k x^k,$$ 
where
\be\label{eqn:bk}
b_k=\sum\limits_{m=k+2}^\infty mq_m .
\ee
Recall that \eqref{eqn:QminusxCrit} shows that $g(x)=\sum\limits_{k=0}^\infty a_k x^k$, where 
\be\label{eqn:ak}
a_k= \sum\limits_{m=k+2}^\infty \!(m-k-1)q_m=b_k-c_k 
\quad\text{ with }\quad c_k=(k+1)\!\!\sum\limits_{m=k+2}^\infty \!q_m 
\ee
and $b_k$ as defined in \eqref{eqn:bk}. 
Equations \eqref{eqn:bk} and \eqref{eqn:ak} yield
\be\label{eqn:Qx1xSums}
{Q(x)-x \over (1-x)(1-Q'(x))}={(1-x)g(x) \over 1-Q'(x)}={\sum\limits_{k=0}^\infty a_k x^k \over ~\sum\limits_{k=0}^\infty b_k x^k}=1-{\sum\limits_{k=0}^\infty c_k x^k \over ~\sum\limits_{k=0}^\infty b_k x^k}.
\ee
The infinite second moment condition implies 
$\lim\limits_{x \rightarrow 1-}\sum\limits_{k=0}^\infty b_k x^k=\infty$ and $\lim\limits_{x \rightarrow 1-}\sum\limits_{k=0}^\infty c_k x^k=\infty$.
As \eqref{eqn:RegCondLim} postulates that $\,\lim\limits_{k \rightarrow \infty}(c_k/b_k)=\Lambda$, 
for a given small value of $\epsilon>0$, there exists $K \in \mathbb{N}$ such that $\,|c_k/b_k-\Lambda|<\epsilon$, $\forall k\geq K$.
Thus
$$\liminf\limits_{x \rightarrow 1-}{\sum\limits_{k=0}^\infty c_k x^k \over ~\sum\limits_{k=0}^\infty b_k x^k}
=\liminf\limits_{x \rightarrow 1-}{\sum\limits_{k=K}^\infty c_k x^k \over ~\sum\limits_{k=K}^\infty b_k x^k} \geq \Lambda-\epsilon$$
and
$$\limsup\limits_{x \rightarrow 1-}{\sum\limits_{k=0}^\infty c_k x^k \over ~\sum\limits_{k=0}^\infty b_k x^k}
=\limsup\limits_{x \rightarrow 1-}{\sum\limits_{k=K}^\infty c_k x^k \over ~\sum\limits_{k=K}^\infty b_k x^k} \leq \Lambda+\epsilon.$$
Consequently, 
$$\lim\limits_{x \rightarrow 1-}{\sum\limits_{k=0}^\infty c_k x^k \over ~\sum\limits_{k=0}^\infty b_k x^k}=\Lambda.$$
Hence, the limit in \eqref{eqn:RegAsmQ} exists, and \eqref{eqn:Qx1xSums} implies 
$$S'(1)=1-\lim\limits_{x \rightarrow 1-}{Q(x)-x \over (1-x)\big(1-Q'(x)\big)}
=\lim\limits_{x \rightarrow 1-}{\sum\limits_{k=0}^\infty c_k x^k \over ~\sum\limits_{k=0}^\infty b_k x^k~}=\Lambda.$$
\end{proof}

\medskip
\begin{cor}[{{\bf Zipf distribution}}]\label{cor:Zipf}
Consider a critical Galton-Watson process $\mathcal{GW}(\{q_k\})$ with $q_1=0$ and offspring distribution $\{q_k\}$ of Zipf type:
\be\label{eqn:ZipfTail}
q_k \sim Ck^{-(\alpha+1)} \quad \text{ with }~\alpha \in (1,2]~\text{ and }~C>0.
\ee 
Then Assumption \ref{asm:reg} is satisfied,
\be\label{eqn:ZipfTailS1g}
S'(1)={\alpha-1 \over \alpha} \quad \text{ and }\quad
L=\lim\limits_{x \rightarrow 1-}\left({\ln{g(x)} \over -\ln(1-x)}\right)=2-\alpha.
\ee 
\end{cor}
\begin{proof}
Suppose $q_k=Ck^{-(\alpha+1)}\big(1+o(1)\big)$. Then, 
$$\sum\limits_{m=k}^\infty mq_m =C{k^{1-\alpha} \over \alpha-1}\big(1+o(1)\big)
\quad \text{ and }\quad
\sum\limits_{m=k}^\infty q_m=C{k^{-\alpha} \over \alpha}\big(1+o(1)\big).$$
Hence, the limit $\Lambda$ defined in \eqref{eqn:RegCondLim} exists and is equal to
$$\Lambda=\lim\limits_{k \rightarrow \infty}{~k\sum\limits_{m=k}^\infty \!q_m ~ \over \sum\limits_{m=k}^\infty \!mq_m}={\alpha-1 \over \alpha}$$
Consequently, Lem.~\ref{lem:RegCond} implies Assumption \ref{asm:reg} and $\,S'(1)={\alpha-1 \over \alpha}$.
Finally, by Lem.~\ref{lem:dSof1} we have 
\[L=2-{1 \over 1-S'(1)}=2-\alpha.\]
\end{proof}

\begin{ex}[{\bf Infinite second moment, $L=0$}]\label{ex:k3Zipf}
Consider a critical Galton-Watson process $\mathcal{GW}(\{q_k\})$ with $q_0={2 \over 3}$, $q_1=0$, and 
$$q_k={4/3 \over k(k^2-1)} \qquad  (k\geq 2).$$
Observe that the offspring distribution $q_k$ is of Zipf type \eqref{eqn:ZipfTail} with $\alpha=2$.
This offspring distribution has infinite second moment.
Here,
$$Q(x)-x=(1-x)^2\,g(x) \quad \text{ with }\quad g(x)=-{2/3 \over x}\ln(1-x),$$
and therefore, the limit in \eqref{eqn:RegAsmQ} exists and is equal to
$$\lim\limits_{x \rightarrow 1-}{Q(x)-x \over (1-x)\big(1-Q'(x)\big)}=\lim\limits_{x \rightarrow 1-}{\ln(1-x) \over 2\ln(1-x)+{1 \over x}}={1 \over 2}.$$
Hence, Assumption \ref{asm:reg} is satisfied with $\,S'(1)=1-{1 \over 2}={1 \over 2}$.
On the other hand,
$$L=\lim\limits_{x \rightarrow 1-}\left({\ln{g(x)} \over -\ln(1-x)}\right)=\lim\limits_{x \rightarrow 1-}\left({\ln\left(-\ln(1-x)\right) \over -\ln(1-x)}\right)=0,$$
which is consistent with Lem.~\ref{lem:2plusMoment}.
We also see that $S'(1)={1-L \over 2-L}={\alpha-1 \over \alpha}$, giving an example for statements in Lem.~\ref{lem:dSof1} and Cor.~\ref{cor:Zipf}.
\end{ex}

\subsection{Tokunaga coefficients in recursive form}

Here we derive a recursive expression for the Tokunaga coefficients of a Galton-Watson measure
in the form $T_{i,j} = \pi_i f(\sigma_{j-2},\pi_{j-1},\pi_j)$.
The recursive nature of this representation is connected to the recursive 
expression \eqref{eqn:GWtreePIj} for $\pi_i$ of Lem.~\ref{lem:GWtreePIj}.

\begin{lem}[{\bf Tokunaga coefficients}]\label{lem:GWtreeTij}
Consider a Galton-Watson measure $\mathcal{GW}(\{q_k\})$ with $q_1=0$.
Assume criticality or subcriticality, i.e., $\sum\limits_{k=0}^\infty kq_k \leq 1$. Then, for all $1\leq i < j-1$, we have
\be\label{eqn:GWtreeTij}
T_{i,j}=\pi_i{Q'(\sigma_{j-1})-Q'(\sigma_{j-2})-\pi_{j-1}Q''(\sigma_{j-2}) \over Q(\sigma_{j-1})-Q(\sigma_{j-2})-\pi_{j-1}Q'(\sigma_{j-2})} +T_{i,j}^o, 
\ee 
and for $1\leq i=j-1$,
\be\label{eqn:GWtreeTj1j}
T_{j-1,j}={\pi_{j-1}Q'(\sigma_{j-1})\!+\!\pi_{j-1}Q'(\sigma_{j-2})\!-\!2Q(\sigma_{j-1})\!+\!2Q(\sigma_{j-2}) \over Q(\sigma_{j-1})-Q(\sigma_{j-2})-\pi_{j-1}Q'(\sigma_{j-2})}
+T_{j-1,j}^o,
\ee 
where $\,T_{i,j}^o=\pi_i {Q''(\sigma_{j-1}) \over 1-Q'(\sigma_{j-1})}\,$ is the expected number of offsprings of order $i$ descendant to all regular (non-terminal) vertices of order $j$.
\end{lem}

\noindent
Note that \eqref{eqn:GWtreeTij} can be rewritten as
$$T_{i,j}=\pi_i {d \over dx}\ln\left({Q(x+\pi_{j-1})-Q(x)-\pi_{j-1}Q'(x) \over 1-Q'(x+\pi_{j-1})}\right)\Bigg|_{x=\sigma_{j-2}}.$$
\begin{proof}
For $i\leq j-2$, let $M_{i,j}^{\rm term}$ denote the expected number of descendants
of order $i$ of a terminal vertex of order $j$, i.e., the expected number of side branches of Tokunaga index $\{i,j\}$; see Remark \ref{rem:side-branching}.
For $d \in \mathbb{N}$, consider all vertices in generation $d$.
The probability that a vertex is a terminal vertex in a branch of order $j$ is
$$\sum\limits_{m=2}^\infty q_m \sum\limits_{\ell=2}^m \binom{m}{\ell} \pi_{j-1}^\ell \sigma_{j-2}^{m-\ell}=Q(\sigma_{j-1})-Q(\sigma_{j-2})-\pi_{j-1}Q'(\sigma_{j-2}),$$
where $m\geq 2$ is its branching number (i.e., the number of descendants) and $\ell \geq 2$ is the number of descendants of order $j-1$.

\medskip
\noindent
Recall that $\sum\limits_{m=0}^k m \binom{k}{m} a^m b^{k-m}=ka(a+b)^{k-1}$. 
The expected number of offsprings of order $i$ descendant to a vertex conditioned 
on having a total of $m\geq 2$ offsprings, of which $\ell \geq 2$ are of order 
$j-1$ and $m-\ell$ are of order smaller than $j-1$, is
$${1 \over \sigma_{j-2}^{m-\ell}}\sum\limits_{k=0}^{m-\ell}k\binom{m-\ell}{k} \pi_i^k \big(\sigma_{j-2}-\pi_i\big)^{m-\ell-k}=\pi_i{m-\ell \over \sigma_{j-2}}.$$
Thus, for $i\leq j-2$,
\begin{align}\label{eqn:termNij}
M_{i,j}^{\rm term}&={\sum\limits_{m=2}^\infty q_m \sum\limits_{\ell=2}^m \binom{m}{\ell} \pi_{j-1}^\ell \sum\limits_{k=0}^{m-\ell}k\binom{m-\ell}{k} \pi_i^k \big(\sigma_{j-2}-\pi_i\big)^{m-\ell-k}
\over \sum\limits_{m=2}^\infty q_m \sum\limits_{\ell=2}^m \binom{m}{\ell} \pi_{j-1}^\ell \sigma_{j-2}^{m-\ell}} \nonumber \\
&\qquad\qquad ={\pi_i \sum\limits_{m=2}^\infty q_m \sum\limits_{\ell=2}^m (m-\ell) \binom{m}{\ell} \pi_{j-1}^\ell \sigma_{j-2}^{m-\ell-1}
\over Q(\sigma_{j-1})-Q(\sigma_{j-2})-\pi_{j-1}Q'(\sigma_{j-2})} \nonumber \\
&\qquad\qquad ={\pi_i \sum\limits_{m=2}^\infty q_m \left(m\sigma_{j-1}^{m-1}-m(m-1)\pi_{j-1}\sigma_{j-2}^{m-2}-m\sigma_{j-2}^{m-1} \right)
\over Q(\sigma_{j-1})-Q(\sigma_{j-2})-\pi_{j-1}Q'(\sigma_{j-2})} \nonumber \\
&\qquad\qquad =\pi_i{Q'(\sigma_{j-1})-Q'(\sigma_{j-2})-\pi_{j-1}Q''(\sigma_{j-2}) \over Q(\sigma_{j-1})-Q(\sigma_{j-2})-\pi_{j-1}Q'(\sigma_{j-2})}.
\end{align}
Next, for $i=j-1$, let $M_{j-1,j}^{\rm term}$ denote the expected number of order $j-1$ side-branches adjacent to a terminal vertex of a branch of order $j$. 
The expected number of order $j-1$ offsprings of a vertex conditioned on being the terminal vertex in a branch of order $j$ with 
a total of $m\geq 2$ offsprings is
$${1 \over \sigma_{j-2}^{m-\ell}}\sum\limits_{\ell=2}^m \ell \binom{m}{\ell} \pi_{j-1}^\ell \sigma_{j-2}^{m-\ell},$$
where $\ell \geq 2$ counts the offspring of order $j-1$, and the rest $m-\ell$ represent the offsprings of order smaller than $j-1$.
Following Remark \ref{rem:side-branching}, we subtract two principal branches from the number of order $j-1$ offsprings. Consequently,
the expected number of order $j-1$ side branches adjacent to a vertex conditioned on being the terminal vertex in a branch of order $j$ with 
a total of $m\geq 2$ offsprings is equal to
$${1 \over \sigma_{j-2}^{m-\ell}}\sum\limits_{\ell=2}^m (\ell-2) \binom{m}{\ell} \pi_{j-1}^\ell \sigma_{j-2}^{m-\ell}.$$
Here, of $\ell \geq 2$ offspring of order $j-1$, two are principle branches and $\ell-2$ are side branches.
Hence, we have
\begin{align}\label{eqn:termNj1j}
M_{j-1,j}^{\rm term}&={\sum\limits_{m=2}^\infty q_m \sum\limits_{\ell=2}^m (\ell-2) \binom{m}{\ell} \pi_{j-1}^\ell \sigma_{j-2}^{m-\ell}
\over Q(\sigma_{j-1})-Q(\sigma_{j-2})-\pi_{j-1}Q'(\sigma_{j-2})} \nonumber \\
&\qquad\qquad ={\sum\limits_{m=2}^\infty q_m \left(m\pi_{j-1}\sigma_{j-1}^{m-1}+m\pi_{j-1}\sigma_{j-2}^{m-1}-2\sigma_{j-1}^m +2\sigma_{j-2}^m\right)
\over Q(\sigma_{j-1})-Q(\sigma_{j-2})-\pi_{j-1}Q'(\sigma_{j-2})} \nonumber \\
&\qquad\qquad={\pi_{j-1}Q'(\sigma_{j-1})+\pi_{j-1}Q'(\sigma_{j-2})-2Q(\sigma_{j-1})+2Q(\sigma_{j-2}) \over Q(\sigma_{j-1})-Q(\sigma_{j-2})-\pi_{j-1}Q'(\sigma_{j-2})}.
\end{align}

\medskip
\noindent
The expected number $V_j^o$ of regular (non-terminal) vertices in a branch of order $j$ is computed as follows:
\be\label{eqn:regVj}
V_j^o={\sum\limits_{r=0}^\infty r \left(\sum\limits_{k=2}^\infty q_k k\sigma_{j-1}^{k-1}\right)^r
\over \sum\limits_{r=0}^\infty  \left(\sum\limits_{k=2}^\infty q_k k\sigma_{j-1}^{k-1}\right)^r}
={Q'(\sigma_{j-1}) \over 1-Q'(\sigma_{j-1})},
\ee
where, following \eqref{eqn:RegVertexQp}, the probability of a vertex being a regular vertex in a branch of order $j$, conditioned on it being of order $j$,  
equals 
$$\sum\limits_{k=2}^\infty q_k k\sigma_{j-1}^{k-1}.$$

\medskip
\noindent
Finally, the expected number $M_{i,j}^o$ of order $i$ offsprings (and therefore, side branches of Tokunaga index $\{i,j\}$) 
in a regular (non-terminal) vertex on a branch of order $j$ is
 \begin{align}\label{eqn:regNij}
M_{i,j}^o &={1 \over \sum\limits_{k=2}^\infty q_k k\sigma_{j-1}^{k-1}} \sum\limits_{k=0}^\infty  q_k k\!\sum\limits_{s=0}^{k-1}s\binom{k-1}{s}\pi_i^s\big(\sigma_{j-1}-\pi_i\big)^{k-1-s} \nonumber \\
&={1 \over Q'(\sigma_{j-1})}\pi_i \sum\limits_{k=2}^\infty  q_k k(k-1)\sigma_{j-1}^{k-2} =\pi_i {Q''(\sigma_{j-1}) \over Q'(\sigma_{j-1})}
\end{align}
for $1\leq i < j$. Here, $k$ counts the total number of offsprings, of which  we have $k$ choices for the offspring of order $j$. Of the remaining $k-1$ ofsprings, we select $s$ offsprings of order $i$ and $k-1-s$ of order other than $i$, but less than $j$. There are $\binom{k-1}{s}$ such choices, 
with probability of $\pi_i^s\big(\sigma_{j-1}-\pi_i\big)^{k-1-s}$ for each such outcome.

\medskip
\noindent
The statement of the lemma follows from equations \eqref{eqn:termNij}, \eqref{eqn:termNj1j}, \eqref{eqn:regVj}, \eqref{eqn:regNij} as
$T_{i,j}=M_{i,j}^{term}+T_{i,j}^o$ with $T_{i,j}^o=V_j^o M_{i,j}^o$ by Wald's equation.
\end{proof}

\begin{ex}[{\bf Critical binary Galton-Watson tree}]\label{cor:ac12inGW}
Consider the critical binary Galton-Watson distribution $\mathcal{GW}(q_0\!=q_2\!=\!1/2)$.
We have
\[Q(z)=\frac{1+z^2}{2},\quad S(z) = \frac{1+z}{2}, \quad\text{and}\quad g(z)=1/2.\]
Corollary \ref{cor:GWtreePIj} yields
$\sigma_j=S(\sigma_{j-1})$ with $\sigma_0=0$, which implies by induction 
$\sigma_j=1-2^{-j}$ and $\quad \pi_j=2^{-j}$ for $j\ge 1.$
Equations \eqref{eqn:GWtreeTij} and \eqref{eqn:GWtreeTj1j} give
\[T_{i,j}=T^o_{i,j}={\pi_i \over 1-\sigma_{j-1}}=2^{j-i-1}\quad\text{for all}\quad 1\leq i < j,\]
which implies the Toeplitz property (Def.~\ref{Tsi}) and Tokunaga self-similarity (Def.~\ref{ss2}) with $(a,c) = (1,2)$ and 
$T_k = 2^{k-1}$.
\end{ex}

\medskip
\begin{lem}[{{\bf Toeplitz implies criticality}}]\label{lem:q0}
Consider a subcritical or critical Galton-Watson measure 
$\mathcal{GW}(\{q_k\})$ with $q_1=0$ that
satisfies Assumption \ref{asm:reg}.
If the Toeplitz property (Def.~\ref{Tsi}) is satisfied, then the measure
is either critical or $q_0=1$, the order distribution is geometric with $\pi_k = q_0(1-q_0)^{k-1}$,
and  $q_0=1-S'(1)$.
\end{lem}
\begin{proof}
The Toeplitz property implies the existence of the Tokunaga sequence $\{T_k\}_{k \in\mathbb{N}}$.

In the trivial case of $q_0=1$, we have $T_k=0$ for any $k\ge1$, $Q(z)=S(z)=1$ so 
$S'(1)=0 = 1-q_0$, and $\pi_k = \delta_{1k}$. This establishes the statement.

Suppose that $q_0<1$.
Equation \eqref{eqn:GWtreeTij} shows that there is a scalar $c>0$ such that
$${T_{k+1} \over T_k}={\pi_i \over \pi_{i+1}}=c \qquad \forall k \geq 2, \, i \geq 1.$$
Thus, as $\pi_1=q_0$, we have $\pi_j=q_0 c^{1-j}$ and since $\sum_j \pi_j=1$ then $c=(1-q_0)^{-1}$.

\medskip
\noindent
Next, observe that since $S(x)=S(1)+S'(1)(x-1)+o(1-x)$ and $S(1)=1$, we have
$$1-q_0={\pi_{i+1} \over \pi_i}={S(\sigma_i)-S(\sigma_{i-1}) \over  \pi_i}
={S'(1)(\sigma_i-\sigma_{i-1})+o(1-\sigma_{i-1}) \over  \pi_i} \rightarrow S'(1)~\text{ as }i \rightarrow \infty$$
that leads to
\begin{equation}\label{eqn:Sof1viaq0}
q_0=1-S'(1).
\end{equation}
The criticality follows from the constraint $q_0<1$, since in the subcritical case we have $S'(1)=0$ (see Rem.~\ref{rem:subcritical}).
\end{proof}

The following statement gives an alternative proof to one of the main results of Burd et al.~\cite{BWW00} using the framework of the present study.
\begin{cor}\label{lem:GWtreeTkGWhalf}
Consider a subcritical or critical offspring distribution $\{q_k\}$ with  
$q_1=0$ and a finite second moment, $\sum_{k=1}^{\infty}k^2q_k<\infty$.
The measure $\mathcal{GW}(\{q_k\})$ satisfies the Toeplitz property (Def.~\ref{Tsi})
if and only if it is the critical binary Galton-Watson measure, $q_0=q_2={1 \over 2}$.
\end{cor}
\begin{proof}
By Lem.~\ref{lem:2ndMomentAsm}, the finite second moment implies 
Assumption \ref{asm:reg}
with $q_0=1-S'(1)={1 \over 2}$.
Assume the Toeplitz property holds. 
Then the criticality follows from Lem.~\ref{lem:q0}.
The criticality with $q_0={1 \over 2}$ and $q_1=0$ yield 
$q_2={1 \over 2}$ as $$\sum\limits_{k=2}^\infty {k \over 2}q_k={1 \over 2}=1-q_0=\sum\limits_{k=2}^\infty q_k.$$
Conversely, the Toeplitz property for the critical binary Galton-Watson tree 
is established in Ex.~\ref{cor:ac12inGW}.
\end{proof}

\subsection{Invariant Galton-Watson measures}
The following result was originally proved in \cite{BWW00}. 
We state and prove it here since the expression \eqref{eqn:genQone} will be used in
the proof of Thm.~\ref{thm:completeGW} below.
\begin{lem}[{{\bf Pruning Galton-Watson tree, \cite{BWW00}}}]
\label{lem:BWW00_1}
Consider a critical or subcritical Galton-Watson measure $\mu\equiv\mathcal{GW}(\{q_k\})$  with $q_1=0$ on $\cT^|$ with generating function $Q(z)$,
and the corresponding pushforward probability measure induced by the Horton pruning operator $\cR$,
$$\nu(T)=\mu \circ \cR^{-1}(T) = \mu \big(\cR^{-1}(T)\big).$$
Then, $\nu(T\,|T\not= \phi)$ is a Galton-Watson measure $\mathcal{GW}(\{q_k^{(1)}\})$ on $\cT^|$ with
offspring probabilities 
\be\label{eqn:q0par0}
q_0^{(1)}={Q(q_0)-q_0 \over (1-q_0)\big(1-Q'(q_0)\big)}, 
\ee
$q_1^{(1)}=0$, and 
\be\label{eqn:q0parm}
q_k^{(1)}={(1-q_0)^{k-1}Q^{(k)}(q_0) \over k!\big(1-Q'(q_0)\big)} \quad (k \geq 2),
\ee
and generating function
\be\label{eqn:genQone}
Q_1(z)=z+{Q\big(q_0+(1-q_0)z\big)-q_0 -z(1-q_0) \over (1-q_0)\big(1-Q'(q_0)\big)}.
\ee
Moreover, if $\mu(T)$ is critical, then so is  $\nu(T\,|T\not= \phi)$.
If $\mu(T)$ is subcritical, then the first moment is decreasing with pruning, i.e.,
$~\sum\limits_{k=2}^\infty kq_k^{(1)} < \sum\limits_{k=2}^\infty kq_k <1$.
\end{lem}

\begin{proof}
The standard thinning argument (with $\pi_1=q_0$ being the probability of eliminating an offspring) implies that $\cR(T)$ is distributed as a Galton-Watson tree, i.e., $\cR(T)\stackrel{d}{\sim}\mathcal{GW}(\{q_m^{(1)}\})$. Indeed, think of a random tree obtained as a result of the auxiliary branching process defined in the following way. 
We trace the branching process that starts with one generation zero progenitor vertex (the root) that produces exactly one offspring.
From generation one on, the branching process evolves according to the offspring distribution $\left\{{q_k \over 1-q_0}\right\}_{k=2,3,\hdots}$.
Next, the process is thinned: once an offspring is produced (in each generation, including generation zero), 
it is either instantaneously eliminated with probability $q_0$ or is left untouched with probability $1\!-\!q_0$, where these 
Bernoulli trials are performed independently of each other and the branching history.
Naturally, this generates a Galton-Watson branching process with branching probabilities $\{p_m\}$ calculated as follows
\be\label{eqn:pmBinom}
p_m=\sum\limits_{k = m\vee 2}^\infty \binom{k}{m}q_0^{k-m} (1-q_0)^m  {q_k \over 1-q_0}.
\ee
The above defined {\it thinned Galton-Watson process} can be equivalently formulated by tracking the original branching process with branching probabilities $\{q_k\}$.
Here, for each offspring, it is instantaneously decided whether the offspring is a leaf or not via a Bernoulli trial with probabilities $q_0$ and $1-q_0$ for `leaf' and `no leaf'
outcomes respectively. If the offspring is decided to be a leaf, it is pruned instantaneously. If not a leaf, it will branch according to the offspring distribution $\left\{{q_k \over 1-q_0}\right\}_{k=2,3,\hdots}$. The thinned Galton-Watson process differs from the original one by pruning all the leaves. Hence, it implements the instantaneous 
Horton pruning, but not yet series reduction. Indeed, the above thinned Galton-Watson prices with branching probabilities $\{p_m\}$ can have single offspring nodes.

\medskip
\noindent
Next, we need to account for the series reduction by generating a Galton-Watson process with the branching probabilities $\{q_m^{(1)}\}$
by letting
$$q_0^{(1)} = {p_0 \over 1-p_1} ~= {(1-q_0)^{-1}\sum\limits_{k = 2}^\infty  q_0^k  q_k \over 1-\sum\limits_{k = 2}^\infty k q_0^{k-1}  q_k},$$
$q_1^{(1)}=0$, and for $m \geq 2$,
$$q_m^{(1)} = {p_m \over 1-p_1} ~={(1-q_0)^{m-1}  \sum\limits_{k = m}^\infty  \binom{k}{m} q_0^{k-m} q_k \over 1-\sum\limits_{k = 2}^\infty k q_0^{k-1} q_k}.$$
This branding process induces the tree measure $\nu(T)$. 
Note that there is an alternative derivation of \eqref{eqn:q0par0} as by Cor.~\ref{cor:GWtreePIj}, $q_0^{(1)}={\pi_2 \over 1-\sigma_1}={S(q_0)-q_0 \over 1-q_0}={Q(q_0)-q_0 \over (1-q_0)(1-Q'(q_0))}$.

\medskip
\noindent
We notice that the corresponding generating function can be computed as follows
\begin{align*}
Q_1(z) & =\sum\limits_{m=0}^\infty z^m q_m^{(1)} 
={(1-q_0)^{-1} \over 1-\sum\limits_{k = 2}^\infty k q_0^{k-1}  q_k}\left(\sum\limits_{k = 2}^\infty  q_0^k  q_k+  \sum\limits_{m=2}^\infty \sum\limits_{k = m}^\infty  \big(zq_0^{-1}(1-q_0)\big)^m \binom{k}{m} q_0^k  q_k \right)
\nonumber \\
&= {(1-q_0)^{-1} \over 1-Q'(q_0)}\left(\sum\limits_{k = 2}^\infty  q_0^k  q_k+   \sum\limits_{k=2}^\infty \sum\limits_{m = 2}^k  \binom{k}{m} \big(zq_0^{-1}(1-q_0)\big)^m q_0^k  q_k \right)
\nonumber \\
&= {(1-q_0)^{-1} \over 1-Q'(q_0)}\left(Q\big(z+(1-z)q_0\big)-q_0 -z(1-q_0)Q'(q_0) \right)
\end{align*}
by the binomial theorem, implying \eqref{eqn:genQone}.
We proceed by differentiating ${d \over dz}$ in \eqref{eqn:genQone}, obtaining
\be\label{eqn:QoneFromQ_0}
Q_1' (z)={Q'(q_0+z(1-q_0))-Q'(q_0) \over 1-Q'(q_0)}.
\ee
Next, we observe that if $\mu(T)$ is critical, \eqref{eqn:QoneFromQ_0} implies $\sum\limits_{k = 2}^\infty k q_k^{(1)}=Q_1' (1)={Q'(1)-Q(q_0) \over 1-Q(q_0)}=1$. 
That is, the critical process stays critical after a Horton pruning.
Finally, in the subcritical case, $Q'(1)<1$, and by formula \eqref{eqn:QoneFromQ_0},
$\,Q_1'(1)={Q'(1)-Q(q_0) \over 1-Q(q_0)}<Q'(1)$.
\end{proof}

\medskip
\noindent
Formula \eqref{eqn:genQone} matches the evolution of the generator under tree erasure discussed by He and Winkel \cite[Lemma 11]{HeWinkel2014}; see also 
Neveu \cite{Neveu86} and Kesten \cite{Kesten87}.
Also, observe that expression \eqref{eqn:genQone} is of the same form as the generating function of a thinned Galton-Watson process
in the work of Duquesne and Winkel \cite[eqn.~(10) of Sec. 2.2]{DW2007}, where the thinning was done in the context of a Bernoulli leaf coloring scheme.

\begin{lem}\label{lem:LimitExists}
Consider a critical or subcritical Galton-Watson measure $\mathcal{GW}(\{q_k\})$ with $q_1=0$.
If it is Horton prune-invariant (self-similar) (Def.~\ref{def:prune}), then the limit 
$$\lim\limits_{x \rightarrow 1-}\left({\ln{g(x)} \over -\ln(1-x)}\right)=L$$
exists and is finite. Moreover, $$L=1-{\ln\big(1-Q'(q_0)\big) \over \ln(1-q_0)}.$$
\end{lem}
\begin{proof} 
The Horton prune-invariance implies $Q_1(z)=Q(z)$ in the recursion \eqref{eqn:genQone}: 
\be\label{eqn:genQone1}
Q(z)=z+{Q\big(q_0+(1-q_0)z\big)-q_0 -z(1-q_0) \over (1-q_0)\big(1-Q'(q_0)\big)},
\ee
which we rewrite as
\be\label{eqn:genQoneRec}
Q\big(q_0+(1-q_0)z\big)-\big(q_0+z(1-q_0)\big) =M(q_0)\big(Q(z)-z\big),   ~\text{ where }~M(q_0)=(1-q_0)\big(1-Q'(q_0)\big).
\ee
Then, for any $k \in \mathbb{N}$,
$$Q\big(1-(1-q_0)^k+(1-q_0)^k z\big)-\big(1-(1-q_0)^k+(1-q_0)^k z\big)=\big(M(q_0)\big)^k \big(Q(z)-z\big)$$
and for $z \in [0,1)$,
\begin{align}\label{eqn:kqLim}
\lim\limits_{k \rightarrow \infty}&{\ln{(Q\big(1-(1-q_0)^k+(1-q_0)^k z\big)-\big(1-(1-q_0)^k+(1-q_0)^k z\big))} \over \ln\Big(1-\big(1-(1-q_0)^k+(1-q_0)^k z\big) \Big)} \nonumber\\
&=\lim\limits_{k \rightarrow \infty}{k\ln M(q_0)+\ln{(Q(z)-z)} \over k\ln(1-q_0)+\ln(1-z)}={\ln M(q_0) \over \ln(1-q_0)}.
\end{align}
Next, notice that for $z \in I_0=[0,q_0)$, 
$$\ln(q_0)\leq\ln{(Q(z)-z)} \leq \ln{\big(Q(q_0)-q_0\big)} ~\text{ and }~\ln(1-q_0) \leq \ln(1-z)\leq 0.$$
Hence, for any $x \in I_k=\big(1-(1-q_0)^k, \,1-(1-q_0)^{k+1}\big)$, there is a $z   \in I_0$ such that
$$x=1-(1-q_0)^k+(1-q_0)^k z$$
and
$${k\ln \!M(q_0)+\!\ln{\!\big(Q(q_0)-q_0\big)} \over k\ln(1-q_0)+\ln(1-q_0)} \leq \!{\ln{(Q(x)-x)} \over \ln(1-x)}\!=\!{k\ln \!M(q_0)\!+\!\ln{(Q(z)-z)} \over k\ln(1-q_0)+\ln(1-z)} 
\leq \!{k\ln \!M(q_0)\!+\!\ln(q_0) \over k\ln(1-q_0)}.$$
Hence, the following limit exists
$$\lim\limits_{x \rightarrow 1-}{\ln{(Q(x)-x)} \over \ln(1-x)}={\ln M(q_0) \over \ln(1-q_0)}.$$
Finally,
\begin{align}\label{eqn:LviaAq0}
\lim\limits_{x \rightarrow 1-}\left({\ln{g(x)} \over -\ln(1-x)}\right)=&\lim\limits_{x \rightarrow 1-}{2\ln(1-x)-\ln{(Q(x)-x)} \over \ln(1-x)}
=2-{\ln M(q_0) \over \ln(1-q_0)} \nonumber \\
&=1-{\ln\big(1-Q'(q_0)\big) \over \ln(1-q_0)}.
\end{align}
\end{proof}

\medskip
\noindent
Next, we define a single parameter family of critical Galton-Watson measures $\mathcal{GW}(\{q_k\})$ with $q_1=0$ on $\cT^|$.
\begin{Def}[{\bf Invariant Galton-Watson measures}]\label{def:IGWq0}
For a given $q \in [1/2,1)$, a critical Galton-Watson measure $\mathcal{GW}(\{q_k\})$ 
is said to be the {\it invariant Galton-Watson (IGW)} measure with parameter $q$ and denoted by $\mathcal{IGW}(q)$ 
if its generating function is given by
\be\label{eqn:completeGWQz}
Q(z)=z+q(1-z)^{1/q}.
\ee
The respective branching probabilities are $q_0=q$, $q_1=0$, $q_2=(1-q)/2q$, and 
\be\label{eqn:completeGWqk}
q_k= {1-q \over k!\,q}\, \prod\limits_{i=2}^{k-1}(i-1/q) \quad (k\geq 3).
\ee
Here, if $q=1/2$, then the distribution is critical binary, i.e., 
$\mathcal{GW}(q_0\!= q_2\!=\!1/2)$.
If $q \in (1/2,1)$, the distribution is of Zipf type with
\be\label{eqn:completeGWqkZipf}
q_k={(1-q) \Gamma(k-1/q) \over q \Gamma(2-1/q) \,k!} \sim C k^{-(1+q)/q}, ~\text{ where }~C={1-q \over q\,\Gamma(2-1/q)}.
\ee
\end{Def}

\begin{thm}[{{\bf Self-similar Galton-Watson measures}}]\label{thm:completeGW}
Consider a critical or subcritical Galton-Watson measure $\mathcal{GW}(\{q_k\})$ with $q_1=0$ that satisfies Assumption \ref{asm:reg}.
The measure is Horton prune-invariant (self-similar) (Def.~\ref{def:prune}) if and only if it is the invariant Galton-Watson (IGW) measure $\mathcal{IGW}(q_0)$ with $q_0\in[1/2,1)$.
\end{thm}

\begin{proof}
Combining equations \eqref{eqn:q0par0} and  \eqref{eqn:genQone}, we have
\be\label{eqn:genQone2}
Q_1(z)=z+q_0^{(1)}{Q\big(q_0+(1-q_0)z\big) \,-\big(q_0+(1-q_0)z\big) \over Q(q_0)-q_0}.
\ee 
If the Galton-Watson measure is Horton prune-invariant, then $Q_1(z)=Q(z)$, and \eqref{eqn:genQone2} implies
\begin{equation*}
R(z)={R\big(q_0+(1-q_0)z\big) \over R(q_0)} \quad \text{ for }\quad R(z)={Q(z)-z \over q_0}
\end{equation*}
for $z \in [0,1)$.
Hence, letting $\ell(z)=\ln R(1-z)$ for $z\in(0,1]$, we have $$\ell(z)+\ell(1-q_0)=\ell\big((1-q_0)z\big).$$ 
Finally, for $r(y)=\ell\big(e^{-y}\big)=\ln R\big(1-e^{-y}\big)$ for $y \in [0,\infty)$ and $\kappa_0=-\ln(1-q_0)$, 
\begin{equation*}
r(y+\kappa_0)=r(y)+r(\kappa_0) \qquad \forall y \in [0,\infty).
\end{equation*}
Therefore,  $r'(y+\kappa_0)=r'(y)$ and $r(y)=-\int\limits_0^y \alpha(w)\,dw$ for some $\kappa_0$-periodic function $\alpha(y)$. 
Thus,
\be\label{eqn:QzviaAlpha}
Q(z)=z+q_0R(z)=z+q_0e^{\ell(1-z)}=z+q_0e^{r\big(-\ln(1-z)\big)}=z+q_0\exp\left\{-\!\!\!\!\!\!\int\limits_0^{-\ln(1-z)} \!\!\!\!\!\!\alpha(w)\,dw\right\}.
\ee
Next, $~0=q_1=Q'(0)=1-\alpha(0) q_0$ implies  $\,\alpha(0) ={1 \over q_0}$. 
Also, for $z \in (0,1)$, $R'(z)={Q'(z)-1 \over q_0}<0$ and $r(y)$ is a decreasing function. Hence, $\alpha(y)>0$ for all $y \in (0,1)$.

\medskip
\noindent
Letting $w=-\ln(1-x)$ in \eqref{eqn:QzviaAlpha}, we have
\be\label{eqn:QzviaAlphaL}
\ln\big(Q(z)-z\big)=\ln(q_0)-\!\!\!\!\!\!\int\limits_0^{-\ln(1-z)} \!\!\!\!\!\!\alpha(w)\,dw ~=\ln(q_0)-\int\limits_0^z {\alpha(-\ln(1-x)) \over 1-x}\,dx \qquad \forall z \in [0,1).
\ee
Recall that $\,{d \over dz}\ln\big(Q(z)-z\big)={-1 \over S(z)-z}$, and therefore,
\be\label{eqn:QzviaSx}
\ln\big(Q(z)-z\big)=\ln(q_0)-\int\limits_0^z {dx \over S(x)-x}  \qquad \forall z \in [0,1).
\ee
Equations \eqref{eqn:QzviaAlphaL} and \eqref{eqn:QzviaSx} yield
\be\label{eqn:SxviaAlphaL}
S(z)=z+q_0(1-z)\varphi(z), \quad \text{ where }\varphi(z)={1 \over q_0 \alpha(-\ln(1-z))}.
\ee
Here, $\,\alpha(0) ={1 \over q_0}$ implies $\varphi(0)=1$. 
Since $\alpha(z)$ is $\kappa_0$-periodic function,
$$\alpha(-\ln(1-z))=\alpha(-\ln(1-z)+\kappa_0)=\alpha\left(-\ln\big(1-(q_0+(1-q_0)z\big)\right)$$
and $\varphi(z)$ satisfies
\be\label{eqn:quasiPhi}
\varphi(z)=\varphi\big(q_0+(1-q_0)z\big) \qquad \forall z \in [0,1).
\ee

\medskip
\noindent
Equation \eqref{eqn:SxviaAlphaL} implies the existence of the limit
$$\varphi(1)=\lim\limits_{x \rightarrow 1-}\varphi(x)={1 \over q_0}\lim\limits_{x \rightarrow 1-}{S(x)-x \over 1-x}={1-S'(1) \over q_0}.$$
Next, iterating \eqref{eqn:quasiPhi}, we have
$$\varphi(x)=\lim\limits_{k \rightarrow \infty}\varphi\left(\big(1-(1-q_0)^k\big)+(1-q_0)^k x\right)=\varphi(1) \qquad \forall x \in [0,1).$$
Hence, $\,\varphi(x)\equiv 1$, and by \eqref{eqn:SxviaAlphaL},
$$S(z)=z+q_0(1-z).$$
Consequently, \eqref{eqn:QzviaSx} implies $Q(z)=z+q_0(1-z)^{1/q_0}$.

\medskip
\noindent
Finally, observe that for an invariant Galton-Watson measure $\mathcal{IGW}(q_0)$ with any $q_0 \in [1/2,1)$ satisfies \eqref{eqn:genQone2}.
In particular, equation \eqref{eqn:completeGWQz} implies
\be\label{eqn:SzIGWq0}
S(z)=q_0+(1-q_0)z.
\ee
The statement of the theorem follows.
\end{proof}

\medskip
\begin{rem}[{{\bf Heuristics for a linear $S(z)$}}]\label{rem:gap1}
Consider a Horton prune-invariant measure (or at least a Toeplitz measure 
with $q_0<1$) that satisfies Assumption~\ref{asm:reg}.
Lemma~\ref{lem:q0} shows that in this case 
$$\frac{\pi_{k+1}}{\pi_{k}} =\frac{\sigma_{k+1}-\sigma_k}{\pi_k}
= S'(1)=1-q_0\quad\text{for all } k\ge 1.$$
Together with the recursion $\sigma_k = S(\sigma_{k-1})$ 
of Cor.~\ref{cor:GWtreePIj} (see also Fig.~\ref{fig:rec}),
this implies that the points $(\sigma_k,S(\sigma_k))$ lie on the line
\[y(z) = q_0+(1-q_0)z.\] 
This observation suggests $S(z) = q_0+(1-q_0)z$ as a possible solution of 
the equation \eqref{eqn:genQone} with $Q_1(z)=Q(z)$, and 
the corresponding $Q(z)=z+q_0(1-z)^{1/q_0}$ is found by \eqref{eqn:QzviaSx}.
Theorem~\ref{thm:completeGW} ensures that this is the only solution
under the regularity Assumption~\ref{asm:reg}. 
\end{rem}

\begin{rem}[{{\bf Intuition behind the regularity condition}}]\label{rem:gap2}
The Horton pruning acts as a rescaling (vertical and horizontal) on the function
 $S(z)-z$ from the restricted domain $[q_0,1]$ to $[0,1]$,
according to \eqref{eqn:genQone}.
After $k$ consecutive prunings, function $S_k(z)-z$ with the domain $[0,1]$
is obtained via scaling from a restriction of $S(z)-z$ to the interval $[1-(1-q_0)^k, \,1]$.
Thus, consecutive pruning rescales and maps the function $S(z)-z$ in the vicinity of $1-$
to the interval $[0,1]$.
Assumption~\ref{asm:reg} requires a smooth behavior of $S(z)$ at $z=1-$.
The rescaling translates this smooth behavior to the ultimate linearity of function $S(z)$
on the entire interval $[0,1]$.
The most general form of prune-invariant $Q(z)$ is given in \eqref{eqn:QzviaAlpha}, 
which allows a non-linear
oscillatory behavior of $S(z)$ between the points $(\sigma_k,S(\sigma_k))$
discussed in Rem.~\ref{rem:gap1}.
The rescaling argument shows that such oscillations necessarily lead 
to non-smooth behavior of $S(z)$ at $z=1-$ and hence 
violate Assumption~\ref{asm:reg}.
\end{rem}

\begin{rem}[{{\bf General prune-invariant measures}}]\label{rem:gapS1L}
Recall that according to Lem.~\ref{lem:LimitExists}, the general 
Horton prune-invariant 
distributions adhere to the existence and finiteness of the  limit 
$L=\lim\limits_{x \rightarrow 1-}\left({\ln{g(x)} \over -\ln(1-x)}\right)$,
which is weaker than $S'(1)$ required in Assumption~\ref{asm:reg} (see Lem.~\ref{lem:dSof1}).
The gap between the two conditions allows 
for the existence of Horton prune-invariant distributions that satisfy \eqref{eqn:genQone1}
and have a nonlinear function $S(z)$. 
An example of such a measure and further discussion is given in Sect.~\ref{sec:dis}.
\end{rem}

\subsection{Attractors and basins of attraction}

\begin{thm}[{{\bf Attraction property of critical Galton-Watson trees}}]\label{thm:IGWattractor}
Consider a critical Galton-Watson measure $\rho_0 \equiv\mathcal{GW}(\{q_k\})$ with $q_1=0$ on $\cT^|$.
Starting with $k=0$, and for each consecutive integer,
let $\nu_k=\cR_*(\rho_k)$ denote the pushforward probability measure induced by the pruning operator, i.e.,
$\nu_k(T)=\rho_k \circ \cR^{-1}(T) = \rho_k \big(\cR^{-1}(T)\big)$,
and set $\rho_{k+1}(T)=\nu_k\left(T~|T\ne\phi\right)$.
Suppose Assumption \ref{asm:reg} is satisfied. 
Then, for any $T\in\cT^|$,
$$\lim_{k\to\infty}\rho_k(T)=\rho^*(T),$$
where $\rho^*$ denotes the invariant Galton-Watson measure 
$\mathcal{IGW}(q)$ with $q=1-S'(1)$.

Finally, if the Galton-Watson measure $\rho_0 \equiv\mathcal{GW}(\{q_k\})$ is subcritical, then 
$\rho_k(T)$ converges to a point mass measure, $\mathcal{GW}(q_0\!=\!1)$.
\end{thm}

\begin{proof}
Let $q_m^{(k)}$ denote the offspring distribution corresponding to the 
critical Galton-Watson tree measure $\rho_k$,
where $q_1^{(k)}=0$ by series reduction. 
First, we observe that  
\begin{align}\label{eqn:q0kFromSx}
\lim\limits_{k \rightarrow \infty} q_0^{(k)}&=\lim\limits_{k \rightarrow \infty}{\pi_k \over 1-\sigma_{k-1}}=\lim\limits_{k \rightarrow \infty}{S(\sigma_{k-1})-\sigma_{k-1} \over 1-\sigma_{k-1}} \nonumber \\
&=\lim\limits_{k \rightarrow \infty}{1+S'(1)\big(\sigma_{k-1}-1\big)+o\big(1-\sigma_{k-1}\big)-\sigma_{k-1} \over 1-\sigma_{k-1}}
~=1-S'(1).
\end{align}

\medskip
\noindent
Let $Q_k(z):=\sum\limits_{m=0}^\infty z^m q_m^{(k)}$ denote the generating function corresponding to the Galton-Watson measure $\rho_k$ and 
$S_k(z)={Q_k(z)-zQ_k'(z) \over 1-Q_k'(z)}$.
Equation \eqref{eqn:genQone} implies
\be\label{eqn:S1vsS}
S_1(z)={1 \over 1-q_0}S\big(q_0+(1-q_0)z\big) -{q_0 \over 1-q_0}.
\ee
For a given $z \in [0,1)$, we iterate \eqref{eqn:S1vsS}, obtaining 
\be\label{eqn:SkvsS}
S_k(z)=\prod\limits_{i=0}^{k-1}{1 \over 1-q_0^{(i)}}\,S\left(\left(1-\prod\limits_{i=0}^{k-1}(1-q_0^{(i)})\right)+z\prod\limits_{i=0}^{k-1}(1-q_0^{(i)})\right) 
+\left(1-\prod\limits_{i=0}^{k-1}{1 \over 1-q_0^{(i)}}\right),
\ee
where $\prod\limits_{i=0}^{k-1}(1-q_0^{(i)}) \leq 2^{-k} \rightarrow 0$ as $k \rightarrow \infty$.
Next, we substitute 
$$S\left(\left(1-\prod\limits_{i=0}^{k-1}(1-q_0^{(i)})\right)+z\prod\limits_{i=0}^{k-1}(1-q_0^{(i)})\right)
=1+(z-1)S'(1)\prod\limits_{i=0}^{k-1}(1-q_0^{(i)})+o\left(\prod\limits_{i=0}^{k-1}(1-q_0^{(i)})\right)$$
into \eqref{eqn:SkvsS}, getting
$$S_k(z)=1+(z-1)S'(1)+o(1).$$
Hence, for a given $z \in [0,1)$, we have
$${d \over dz}\ln\left(Q_k(z)-z\right)={1 \over z-S_k(z)} ~\longrightarrow~ {1 \over (1-S'(1))(z-1)}\qquad \text{ as }~k \rightarrow \infty.$$
Also, we notice that $Q_k(x)-x$ is a decreasing function ($Q'_k(x)<Q'_k(1)=1$) and 
$$q_0^{(k)} \geq Q_k(x)-x \geq Q_k(z)-z >0 \qquad \forall x \in [0,z].$$
Therefore, letting $k \rightarrow \infty$, we have
$$\ln\left(Q_k(z)-z\right)=\ln q_0^{(k)}+\int\limits_0^z {d \over dx}\ln\left(Q_k(x)-x\right)\,dx
~\longrightarrow~  \ln{q}+{1 \over q}\ln(1-z),$$
where $q=1-S'(1)$, as $\lim\limits_{k \rightarrow \infty} q_0^{(k)}=q$ by \eqref{eqn:q0kFromSx}.
We conclude that
$$\lim\limits_{k \rightarrow \infty} Q_k(z)=z+q\,(1-z)^{1/q}$$
where the right hand side is the generating function for 
$\mathcal{IGW}\left(q\right)$.

\bigskip
\noindent
Finally, if $\rho_0\equiv\mathcal{GW}(\{q_k\})$  is subcritical, \eqref{eqn:dSof1subcrit} 
and \eqref{eqn:q0kFromSx} imply $\lim\limits_{k \rightarrow \infty} q_0^{(k)}=1-S'(1)=1$.
\end{proof}

\medskip
\noindent
Theorem \ref{thm:IGWattractor} and Cor.~\ref{cor:Zipf} immediately imply the following result.
\begin{cor}[{{\bf Attraction property of critical Galton-Watson trees of Zipf type}}]\label{cor:SZipfIGWattractor}
Consider a critical Galton-Watson process $\rho_0 \equiv \mathcal{GW}(\{q_k\})$ with $q_1=0$, with offspring distribution $q_k$ of Zipf type, 
i.e., $q_k \sim C k^{-(\alpha+1)}$, with $\alpha \in (1,2]$ and $C>0$.
Starting with $k=0$, and for each consecutive integer,
let $\nu_k=\cR_*(\rho_k)$ denote the pushforward probability measure induced by the pruning operator, 
and set $\rho_{k+1}(T)=\nu_k\left(T~|T\ne\phi\right)$.
Then, for any $T\in\cT^|$,
$$\lim_{k\to\infty}\rho_k(T)=\rho^*(T),$$
where $\rho^*$ is the invariant Galton-Watson measure $\mathcal{IGW}\left({1 \over \alpha}\right)$.
\end{cor}

\medskip
\noindent
Next, Lem.~\ref{lem:2plusMoment} and \ref{lem:2ndMomentAsm} imply the following attraction 
result as a corollary of our Thm.~\ref{thm:IGWattractor}.
The same attraction property has been established in \cite{BWW00} under the
assumption of a bounded offspring distribution.
\begin{cor}[{{\bf Attraction property of critical binary Galton-Watson tree, \cite{BWW00}}}]\label{BWW00_1}
Consider a critical Galton-Watson process $\rho_0 \equiv \mathcal{GW}(\{q_k\})$ with $q_1=0$.
Assume one of the following two conditions holds.
\begin{description}
  \item[(a)] The second moment assumption is satisfied:
  $$\sum \limits_{k=2}^\infty k^2 q_k ~<\infty.$$
  \item[(b)] Assumption \ref{asm:reg} is satisfied, and the ``$2-$'' moment assumption is satisfied, i.e.,
  $$\sum \limits_{k=2}^\infty k^{2-\epsilon}q_k ~<\infty \qquad \forall \epsilon>0.$$
\end{description}
Starting with $k=0$, and for each consecutive integer,
let $\nu_k=\cR_*(\rho_k)$ denote the pushforward probability measure induced by the pruning operator, 
and set $\rho_{k+1}(T)=\nu_k\left(T~|T\ne\phi\right)$.
Then, for any $T\in\cT^|$,
$$\lim_{k\to\infty}\rho_k(T)=\rho^*(T),$$
where $\rho^*$ is the critical binary Galton-Watson measure $\mathcal{IGW}(1/2)$.
\end{cor}

Figure~\ref{fig:attractor} illustrates convergence of a tree with a 
large branching number to a binary tree.

\begin{figure}[t!] 
\centering\includegraphics[width=0.9\textwidth]{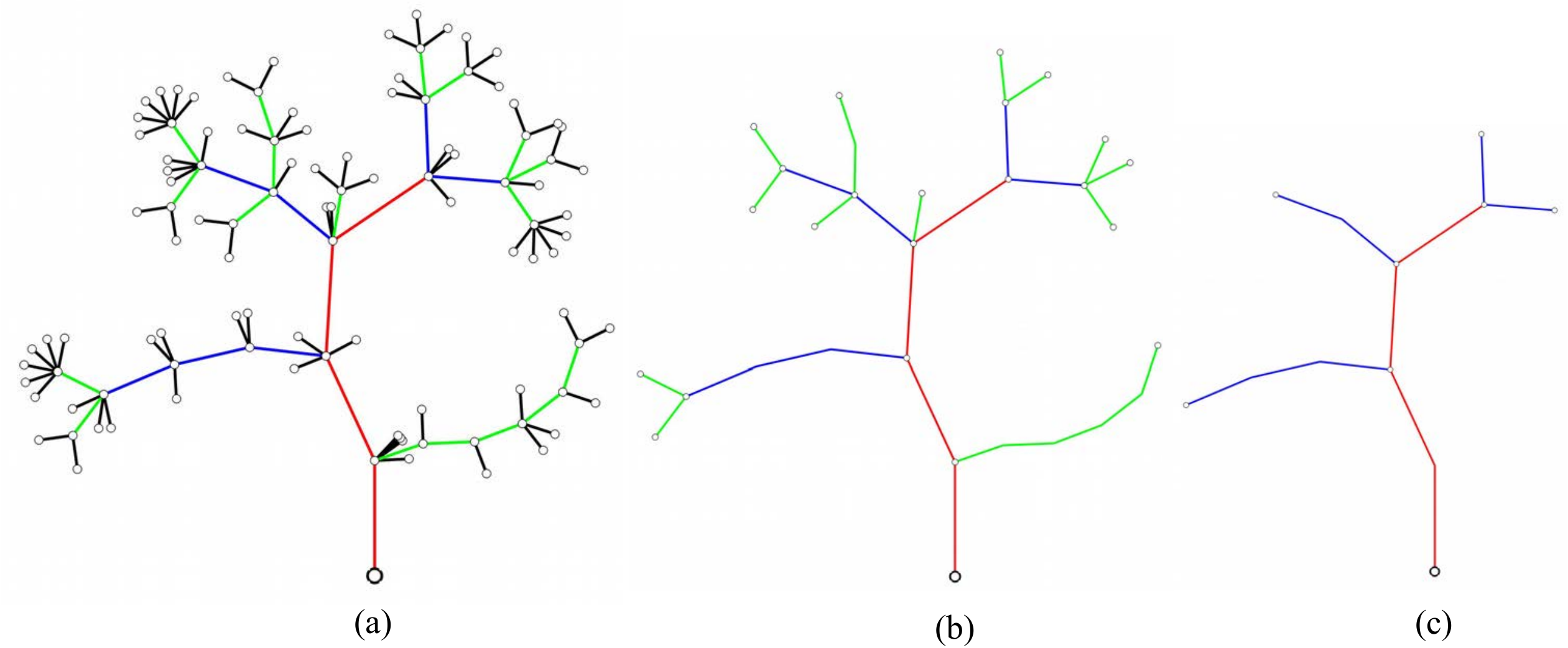}
\caption[Binary attractor: Illustration]
{Binary attractor: Illustration. 
The tree $T$ (panel a) has maximal branching number $b=6$.
Its first pruning (panel b) $\cR(T)$ has maximal branching number $b=3$.
Its second pruning (panel c) $\cR^2(T)$ has maximal branching number $b=2$.
This convergence to binary branching is generic in Galton-Watson 
trees that have offspring distribution with a finite 
$2-\epsilon$ moment; see Cor.~\ref{BWW00_1}.}
\label{fig:attractor}
\end{figure}

%

\subsection{Explicit Tokunaga coefficients and Horton law}
In the next lemma we find the Tokunaga coefficients and the Horton exponent for an invariant Galton-Watson tree measure $\mathcal{IGW}(q_0)$. 
\begin{lem}[{\bf Tokunaga coefficients}]\label{lem:IGWq0Tokunaga}
Consider an invariant Galton-Watson measure $\mathcal{IGW}(q_0)$ for $q_0 \in [1/2,1)$. 
Then,
$$\pi_i=q_0\,c^{1-i} \quad \text{ with }\quad c={1 \over 1-q_0}.$$
The measure satisfies Toeplitz property (Def.~\ref{Tsi}) with the Tokunaga coefficients
\be\label{eqn:IGWq0TokunagaTijo}
T_{i,j}^o=T_{j-i}^o, \quad \text{ where }~~T_k^o=c^{k-1} \quad (k=1,2,\hdots),
\ee
and
\be\label{eqn:IGWq0TokunagaTij}
T_{i,j}=T_{j-i}, \quad \text{ where }~~T_1=c^{c/(c-1)}-c-1 \quad  \text{ and } \quad T_k=a\,c^{k-1} \quad (k=2,3,\hdots)
\ee
with $a=(c-1)\big(c^{1/(c-1)}-1\big)$.
Finally, the strong Horton law \eqref{eq:Horton_mean} holds 
with Horton exponent $\,R=c^{c/(c-1)} = (1-q_0)^{-1/q_0}$.
\end{lem}
\noindent The functions $a(q_0),c(q_0)$ and $R(q_0)$ are illustrated in Fig.~\ref{fig:acR}.
\begin{proof}
Equations \eqref{eqn:GWtreeSigmaj} and \eqref{eqn:SzIGWq0} imply $\,\sigma_i=1-(1-q_0)^i \,+(1-q_0)^i\,z$.
Hence, $\pi_i=\sigma_i-\sigma_{i-1}=q_0 (1-q_0)^{i-1}$.
Equations \eqref{eqn:IGWq0TokunagaTijo} and \eqref{eqn:IGWq0TokunagaTij} are obtained via
substituting $\pi_i$ and $\sigma_i$ into Lem.~\ref{lem:GWtreeTij}.

\medskip
\noindent
Finally, Thm.~\ref{thm:HLSST} implies the strong Horton law 
with the Horton exponent $R=1/w_0$, where $w_0$ is the only real zero of
the generating function $\hat{t}(z)$ in the interval $\left(0,{1 \over 2}\right]$.
We have
$$\hat{t}(z)=-1+(T_1+2)z+{acz^2 \over 1-cz},$$
which gives $w_0=c^{-c/(c-1)}$ and $\,R=c^{c/(c-1)}$. 
\end{proof}

\medskip
\section{Discussion}\label{sec:dis}
In this paper we described the invariance and attractor properties of combinatorial 
Galton-Watson trees with respect to the Horton pruning.
The results hold under the regularity Assumption~\ref{asm:reg} that prohibits 
large tail oscillations of the offspring probabilities $q_m$ that lead to
a non-smooth behavior of $S(z)$ at $1-$.
A sufficient condition under which the regularity
assumption holds is suggested in Lem.~\ref{lem:RegCond}.
Theorem~\ref{thm:completeGW} introduces a one-parameter family of invariant Galton-Watson
distributions $\mathcal{IGW}(q)$ and asserts that this family exhausts the Horton prune-invariant 
distributions within the examined regularity class.
The invariant family has a power-law tail of the offspring distribution, 
$q_k\sim C k^{-\alpha}$, with exactly
one distribution for every $\alpha\in(2,3]$, and also includes the critical
binary Galton-Watson tree measure.

A similar approach can be applied to the search of invariant measures in a 
broader class of {\it generalized dynamical prunings} on trees with edge lengths introduced and analyzed in \cite{KZ19,KZsurvey} and to the pruning operation 
studied in Evans \cite{Evans2005} and in Duquesne and Winkel \cite{Winkel2012}. 
Informally, a generalized dynamical pruning erases a tree from leaves down to the root 
at a rate that only depends on the descendant part of the tree. 
Most of such prunings, with a notable exception of the Horton pruning and the continuous
erasure of Neveu \cite{Neveu86}, do not satisfy semigroup property. 
It has been shown in \cite{KZ19} that the critical binary Galton-Watson tree 
with i.i.d. exponential edge lengths is invariant with respect to any admissible
generalized dynamical pruning. 
We conjecture that the Galton-Watson trees that have i.i.d. exponential edge
lengths and combinatorial shapes sampled from the 
invariant Galton-Watson measures $\mathcal{IGW}(q)$ 
introduced in this work (Def.~\ref{def:IGWq0}) are the only Galton-Watson measures 
invariant with respect to all admissible generalized dynamical prunings, up to
rescaling of the edge lengths.
Heuristically, this is supported by the rescaling argument (Rem.~\ref{rem:gap2})
applied to the function $S(z)$.
The Horton pruning of the present work only requires linearity of 
the function $S(z)$ on the grid $\sigma_k$, which is related to its discrete
combinatorial action.
This allows the existence of prune-invariant measures with oscillatory behavior,
outside of the invariant Galton-Watson family. 
However, a continuous pruning, for instance the continuous erasure of Neveu~\cite{Neveu86},
would constrain the function $S(z)$ on the entire interval $[0,1]$, hence leading to
the family of invariant Galton-Watson trees. 
This will be explored in a follow-up paper. 

\medskip
We are grateful to the anonymous referee for finding a problem with the first version of this paper, caused by the gap consisting of all critical Galton-Watson measures for which the limit $L=\lim\limits_{x \rightarrow 1-}\left({\ln{g(x)} \over -\ln(1-x)}\right)$ exists while the limit $S'(1)=\lim\limits_{x \rightarrow 1-}{1-S(x) \over 1-x}$ does not;
see Rem.~\ref{rem:gapS1L}. 
In particular, the referee suggested the following family of Horton prune-invariant
critical Galton-Watson tree distributions different from the invariant
distributions of Thm.~\ref{thm:completeGW}. 
For a given probability $q_0 \in (1/2,1)$, we let $q_1=0$, and
\be\label{eqn:typeSum}
q_m={1 \over m!A}\sum\limits_{n \in \mathbb{Z}}B^n \rho^{nm}e^{-\rho^n} \qquad m=2,3,\hdots,
\ee
where $\rho=1-q_0$.
Then, the second derivative of the generating function is equal to
$$Q''(z)=\sum\limits_{m=2}^\infty m(m-1)q_m z^{m-2} ={1 \over A}\sum\limits_{n \in \mathbb{Z}}B^n \rho^{2n} e^{-(1-z)\rho^n}, \qquad |z|<1.$$
Observe that
\be\label{eqn:ddQinv}
Q''\big(q_0+(1-q_0)z\big)=B^{-1}\rho^{-2}Q''(z).
\ee
Therefore, if $A>0$ and $B \in \big((1-q_0)^{-1}, \,(1-q_0)^{-2}\big)$ are selected so that
$\,Q(1)=Q'(1)=1$,
then $Q(z)$ will satisfy the invariance criterion \eqref{eqn:genQone1}. 
Such $B$ is found by solving
$$\sum\limits_{n \in \mathbb{Z}}B^n \left(1-\rho^{n+1}- (1+\rho^n-\rho^{n+1})e^{-\rho^n}\right)=0,$$
and $A=\sum\limits_{n \in \mathbb{Z}}B^n \rho^n \left(1- e^{-\rho^n}\right)$.

\medskip
\noindent
Hence, 
$$Q(z)=q_0+{1 \over A}\sum\limits_{n \in \mathbb{Z}}B^n \left(e^{-(1-z)\rho^n}-(1+\rho^n z)e^{-\rho^n}\right), \qquad |z|<1,$$
satisfies \eqref{eqn:genQoneRec} with
$$M(q_0)=(1-q_0)\big(1-Q'(q_0)\big)=B^{-1}.$$
Next, we show that this example belongs to the gap described in Rem.~\ref{rem:gapS1L}. Specifcally, we show that the limit $L$ exists while the limit $S'(1)$ does not.
First, Lem.~\ref{lem:LimitExists} applies, yielding the existence of limit $L=\lim\limits_{x \rightarrow 1-}\left({\ln{g(x)} \over -\ln(1-x)}\right)$.
Moreover, equation \eqref{eqn:LviaAq0} implies
$$L=2-{\ln M(q_0) \over \ln(1-q_0)}=2+{\ln{B} \over \ln(1-q_0)}.$$

\begin{figure}[t] 
\centering\includegraphics[width=0.45\textwidth]{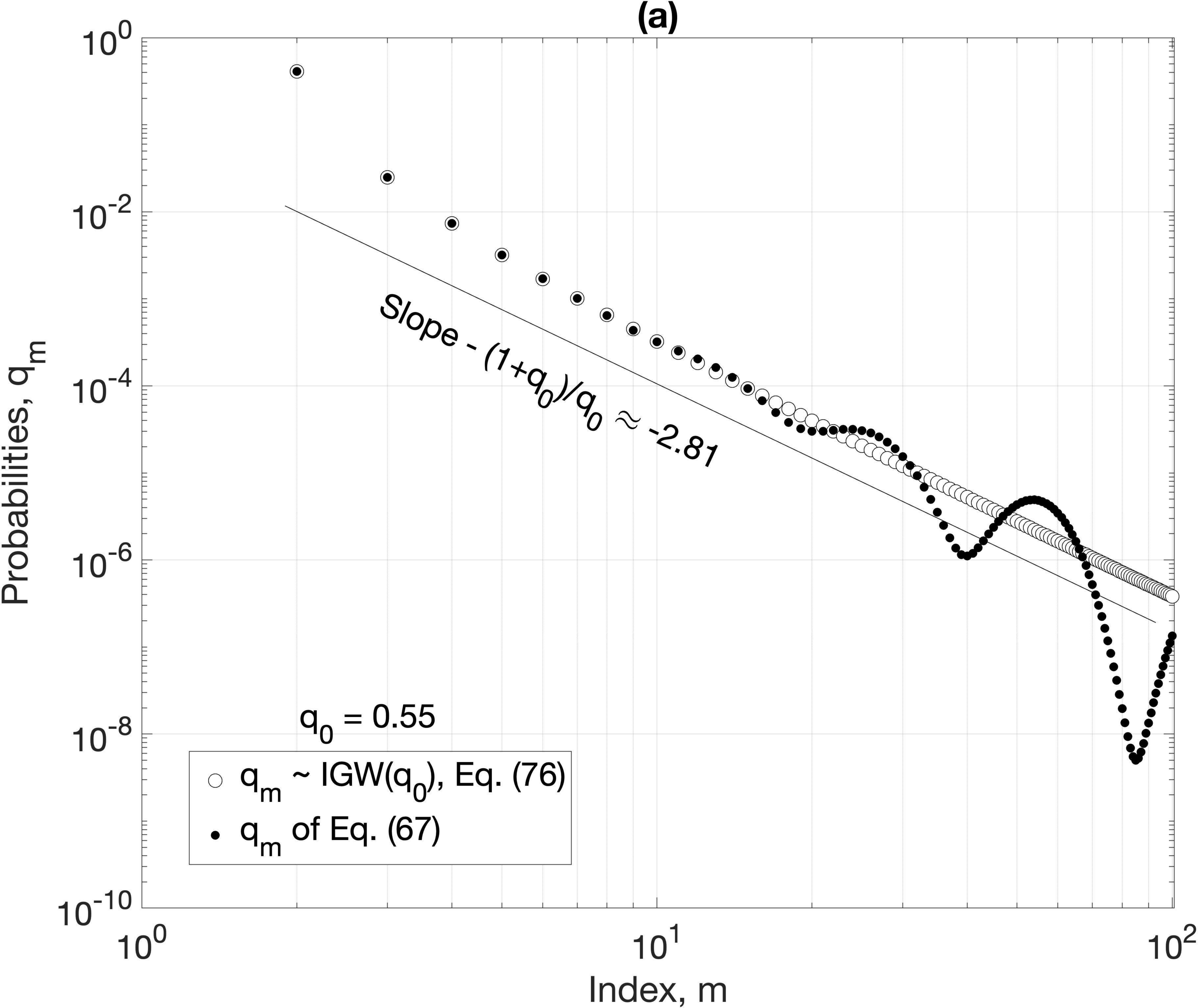}
\centering\includegraphics[width=0.45\textwidth]{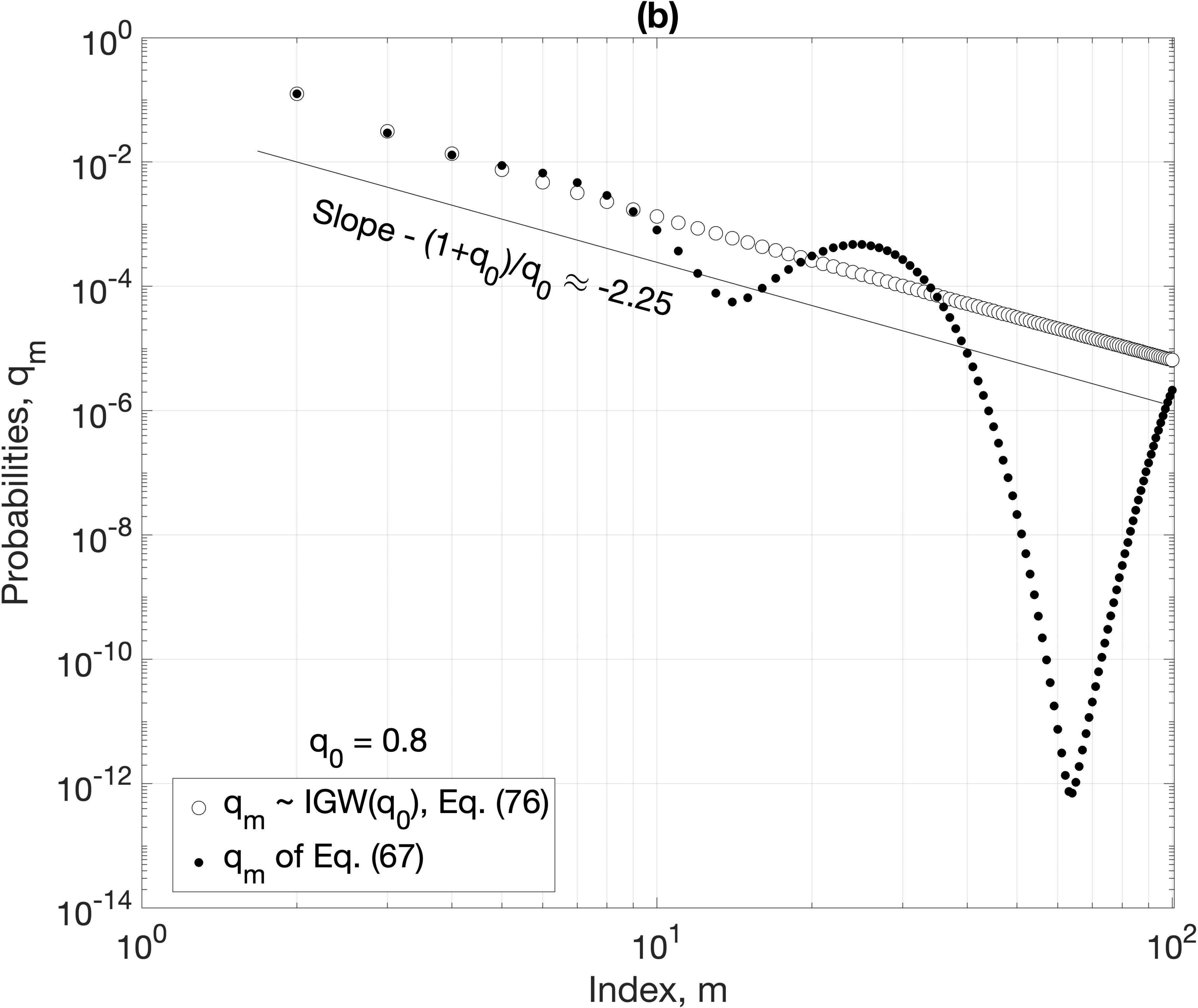}
\caption[Invariant measures: Illustration]
{Horton prune-invariant measures: Illustration.
Figure compares the offspring probabilities $q_m$, $m\ge 2$, of the invariant 
measure $\mathcal{IGW}(q_0)$ 
of Def.~\ref{def:IGWq0}, also given in \eqref{qm_cont} (open circles), with those of the invariant measure of \eqref{eqn:typeSum} that does not satisfy Assumption \ref{asm:reg} (black circles).
(a) $q_0 = 0.55$, (b) $q_0=0.8$.}
\label{fig:ex}
\end{figure}

\medskip
\noindent
Now, we show that the limit $S'(1)=\lim\limits_{x \rightarrow 1-}{1-S(x) \over 1-x}$ does not exist, whence Assumption \ref{asm:reg} is not satisfied.
Since $\,{1 \over A}\sum\limits_{n \in \mathbb{Z}}B^n \rho^n \left(1- e^{-\rho^n}\right)=1$, we have for $x \in [0,1)$,
\be\label{eqn:1Q1xRevEx}
{1-Q(x) \over 1-x}={1 \over A}\sum\limits_{n \in \mathbb{Z}}B^n \rho^n \left(1- e^{-\rho^n}\right)-{1-Q(x) \over 1-x}
={1 \over A}\sum\limits_{n \in \mathbb{Z}}B^n\rho^n \left(1-{1-e^{-(1-x)\rho^n} \over (1-x)\rho^n}\right).
\ee
Also, for $x \in [0,1)$, 
\be\label{eqn:1dQRevEx}
1-Q'(x)={1 \over A}\sum\limits_{n \in \mathbb{Z}}B^n \rho^n \left(1- e^{-(1-x)\rho^n}\right).
\ee
For a given $\alpha \in [0,1)$, consider a sequence $x_m=1-\rho^{m+\alpha}$ for $m \in \mathbb{N}$, then, equation \eqref{eqn:1Q1xRevEx} implies
$${1-Q(x_m) \over 1-x_m}={1 \over A}\sum\limits_{n \in \mathbb{Z}}B^n\rho^n \left(1-{1-e^{-\rho^{n+m+\alpha}} \over \rho^{n+m+\alpha}}\right)
={1 \over A}B^{-(m+\alpha)}\rho^{-(m+\alpha)}C(\alpha),$$
where
\be\label{eqn:CaRevEx}
C(\alpha)=\sum\limits_{n \in \mathbb{Z}}B^{n+\alpha} \rho^{n+\alpha} \left(1-{1-e^{-\rho^{n+\alpha}} \over \rho^{n+\alpha}}\right).
\ee
Similarly, \eqref{eqn:1dQRevEx} implies
$$1-Q'(x_m)={1 \over A}\sum\limits_{n \in \mathbb{Z}}B^n\rho^n \left(1-e^{-\rho^{n+m+\alpha}}\right)
={1 \over A}B^{-(m+\alpha)}\rho^{-(m+\alpha)}D(\alpha),$$
where
\be\label{eqn:DaRevEx}
D(\alpha)=\sum\limits_{n \in \mathbb{Z}}B^{n+\alpha} \rho^{n+\alpha} \left(1-e^{-\rho^{n+\alpha}} \right).
\ee
Hence,
$${1-Q(x_m) \over (1-x_m)\big(1-Q'(x_m)\big)}={C(\alpha) \over D(\alpha)} \qquad \forall m \in \mathbb{N}$$
with $C(\alpha)$ and $D(\alpha)$ as defined in \eqref{eqn:CaRevEx} and \eqref{eqn:DaRevEx}.
Thus, since ${C(\alpha) \over D(\alpha)}$ is not constant for $\alpha \in [0,1)$, the limit in \eqref{eqn:RegAsm} does not exist, and the same is true about the limit $S'(1)$.
Hence, in this example, Thms.~\ref{thm:completeGW} and \ref{thm:IGWattractor} as currently stated do not apply
since Assumption \ref{asm:reg} does not hold.

\bigskip
\noindent
The invariant measures \eqref{eqn:typeSum} are in fact closely related to 
the $\mathcal{IGW}(q)$ measures of Def.~\ref{def:IGWq0}.
Specifically, consider a continuous integral version of \eqref{eqn:typeSum}, by 
selecting $q_0 \in (1/2,1)$ and letting $q_1=0$ and 
\be\label{eqn:typeInt}
q_m={1 \over m!A}\int\limits_{-\infty}^\infty B^w \rho^{wm} e^{-\rho^w}\,dw \qquad m=2,3,\hdots,
\ee 
where $\rho=1-q_0$. Here too,
$$Q''(z)=\sum\limits_{m=2}^\infty m(m-1)q_m x^{m-2}={1 \over A}\int\limits_{-\infty}^\infty B^w \rho^{2w}e^{-(1-x)\rho^w}\,dw$$
and so it satisfies \eqref{eqn:ddQinv}, while $\,Q(1)=Q'(1)=1$ causes  $Q(z)$ to satisfy \eqref{eqn:genQone1}.
Analogously to \eqref{eqn:typeSum}, the constants $A>0$ and $B \in \big((1-q_0)^{-1}, \,(1-q_0)^{-2}\big)$ are found by solving
\be\label{eqn:IntegralC}
\int\limits_{-\infty}^\infty B^w \left(1-\rho^{w+1}- (1+\rho^w-\rho^{w+1})e^{-\rho^w}\right)\,dw=0
\ee
for $B$, and letting
\be\label{eqn:IntegralA}
A=\int\limits_{-\infty}^\infty B^w \rho^w \left(1- e^{-\rho^w}\right)\,dw.
\ee
One can show that equation \eqref{eqn:IntegralC} has a unique solution using monotonicity of the integral on the left hand side in the equation \eqref{eqn:IntegralC} as a function of $B \in \big((1-q_0)^{-1}, \,(1-q_0)^{-2}\big)$. The uniqueness of $B$ implies the uniqueness of $A$ in \eqref{eqn:IntegralA}.

\medskip
\noindent
In this situation, one easily shows that the limit $S'(1)$ exists. 
Remarkably, the unique offspring distribution $q_m$ in \eqref{eqn:typeInt} 
is the offspring distribution of $\mathcal{IGW}(q_0)$.
Indeed, in equation \eqref{eqn:completeGWqkZipf} we have for $m\geq 2$,
\be \label{qm_cont}
q_m={(1-q_0) \Gamma(m-1/q_0) \over q_0 \Gamma(2-1/q_0) \,m!}={1 \over m!A}\int\limits_{-\infty}^\infty B^w \rho^{wm} e^{-\rho^w}\,dw
\ee 
with $\rho=1-q_0$, $~A={q_0 \Gamma(2-1/q_0) \over -(1-q_0)\ln(1-q_0)}$, and $B=(1-q_0)^{-1/q_0}$.

Figure~\ref{fig:ex} compares selected invariant measures $\mathcal{IGW}(q_0)$ of Def.~\ref{def:IGWq0},
also given in \eqref{qm_cont}, with the invariant measures of 
\eqref{eqn:typeSum} that do not satisfy Assumption \ref{asm:reg}.
Both types of measures decay in general as the power law $m^{-(1+q_0)/q_0}$,
although the measures of \eqref{eqn:typeSum} fluctuate around this general trend.
The amplitude of the fluctuations (on logarithmic scale) increases with $q_0$.
These fluctuations are related to the periodic function $\alpha(y)$ in the proof of
Thm.~\ref{thm:completeGW}; they are inevitable in the invariant measures that do not
satisfy Assumption~\ref{asm:reg} (see Rem.~\ref{rem:gap2}).

\medskip
\section*{Acknowledgements}
The authors are grateful to the anonymous referee for finding a gap in the original
version of the paper and suggesting an example of invariant offspring distribution with oscillatory 
tail (see Discussion and Fig.~\ref{fig:ex}) and to Ed Waymire for multiple discussions about the topic of random self-similar trees. 
This research is supported by NSF awards DMS-1412557 (YK) and EAR-1723033 (IZ).


\end{document}